\documentclass[reqno]{amsart}
\usepackage[english]{babel}
\usepackage[utf8]{inputenc}
\usepackage[dvipsnames]{xcolor}
\usepackage{enumerate}
\definecolor{darkgreen}{rgb}{0,0.45,0}
\usepackage[pagebackref, colorlinks,citecolor=darkgreen,linkcolor=darkgreen, urlcolor=darkgreen]{hyperref}
\usepackage[normalem]{ulem}

\usepackage{amsmath, amssymb, amsthm, amsrefs, amscd, amsfonts, mathtools}

\newtheorem{thm}{Theorem}[section]

\newtheorem{prop}[thm]{Proposition}
\newtheorem{theorem}[thm]{Theorem}

\newtheorem{cor}[thm]{Corollary}
\theoremstyle{definition}

\newtheorem{definition}[thm]{Definition}

\newtheorem{exa}[thm]{Example}
\DeclareMathOperator{\spn}{span}

\numberwithin{equation}{section}\theoremstyle{plain}

\newcommand{\B}{\mathcal{B}}

\theoremstyle{definition}

\def \o {\otimes}
\def \k {\Bbbk}

\def \K {\mathcal{K}}
\def \A {\mathcal{A}}

\allowdisplaybreaks

 \title[Partial actions of 8-dimensional Hopf algebras]{One-dimensional partial actions of 8-dimensional Hopf algebras and right coideal subalgebras}
\author[M. Castelo, L. Duarte Silva, W. Hautekiet and G. Martini]{Matheus Castelo, Leonardo Duarte Silva, \\ William Hautekiet and Grasiela Martini}

\address[M. Castelo]{Universidade Federal do Rio Grande do Sul, Brazil}
\email{matheus.castelo@ufrgs.br}

\address[L. Duarte Silva]{Universidade Federal do Rio Grande do Sul, Brazil}
\email{dsleonardo@ufrgs.br}

\address[W. Hautekiet]{Université Libre de Bruxelles, Belgium}
\email{william.hautekiet@ulb.be}

\address[G. Martini]{Universidade Federal do Rio Grande do Sul, Brazil}
\email{grasiela.martini@ufrgs.br}

\begin{document}

\allowdisplaybreaks

\begin{abstract}    
     In this work, we complete the description of the one-dimensional partial actions of 8-dimensional Hopf algebras by computing the remaining cases: the Kac–Paljutkin algebra  $\mathcal{A}$ and the unique non-semisimple non-pointed Hopf algebra $\mathcal{K}$. 
     We prove that all these partial actions are symmetric, study their associated partial smash products and we determine all their partial coactions of dimension one.
     Beyond the 8-dimensional setting, we investigate which right coideal subalgebras of a Hopf algebra $H$ can be realized as partial smash products $\underline{\Bbbk\#H}$ over the base field.
     In particular, we show that every right coideal subalgebra of a finite-dimensional cosemisimple Hopf algebra arises in this way.
\end{abstract}



\thanks{The second and fourth authors were partially supported by FAPERGS (Brazil), projects n. 23/2551-0000913-7 and 23/2551-0000803-3, respectively.}

\maketitle

\tableofcontents

\section{Introduction}
Partial actions are a relatively recent concept in mathematics, first introduced by R. Exel in the context of operator algebras in the 1990s \cite{exel1994circle}.
The idea of studying symmetries that are only partially defined on the objects they act upon is simple yet powerful.
Since its early development and particularly from the early 2000s onwards, when the theory was formulated in an algebraic setting \cite{Dokuchaev_exel}, the concept of partial actions has been widely studied and extended to a several frameworks, and found deep and far-reaching applications across several branches of mathematics such as $C^\ast$-algebras, groupoids, categories and many others, as can be noticed in \cite{Dokuchaev_survey} and the references therein.

In the theory of Hopf algebras, partial actions play a role analogous to that of classical (global) actions, but introduce new challenges and perspectives.
They are closely related to partial group actions, partial smash products, and generalized Galois theories \cites{caenepeel2008partial, paques_ferrero_dokuchaev}, among others.
Although partial actions of group algebras and universal enveloping algebras of Lie algebras have been well understood very early in the theory \cites{Alvares_Alves_Batista, caenepeel2008partial, Dokuchaev_exel, paques_ferrero_dokuchaev}, it has been proven more difficult to classify the partial actions of more general families of Hopf algebras.

One promising strategy to investigate partial actions (or even partial representations) of a given Hopf algebra, or a class of Hopf algebras, is to begin by computing their one-dimensional partial actions.
This approach has proven particularly fruitful for pointed Hopf algebras.
In \cite{corresponding}, a method for computing such one-dimensional partial actions was developed, and applied to obtain them for all non-semisimple pointed Hopf algebras of dimensions 8 and 16.
Furthermore, the notion of a $\lambda$-Hopf algebra was introduced therein.
The $\lambda$-Hopf algebras associated to these partial actions were computed explicitly, providing a complete description in those cases, and showing some perspectives on how the behavior of partial actions of such Hopf algebras is similar to that of group algebras.

The effectiveness of the strategy and of the method proposed in \cite{corresponding} is further illustrated in \cite{FMS}, where it is used to compute the one-dimensional partial actions for two important families of non-semisimple pointed Hopf algebras: the Taft algebras and the Nichols Hopf algebras. 
Moreover, the idea that understanding one-dimensional partial actions is an interesting step towards compute arbitrary partial actions is reinforced in \cites{Arthur, Ore, FMS2}, where the one-dimensional case plays a central role in the analysis.

\medbreak

The present paper has two main objectives: complete the computation of one-dimensional partial actions of 8-dimensional Hopf algebras, and study right coideal subalgebras, in particular those ones that arise from partial smash products over the base field. Recall, as proved in \cite{stefan}, that there are exactly 14 Hopf algebras (up to isomorphism) of dimension 8: of these, 8 are semisimple algebras and 6 are not.
Among the semisimple ones, 7 are either group algebras or duals of group algebras. 
The other semisimple Hopf algebra is the well-known Kac-Paljutkin algebra described in \cite{masuoka}, denoted here by $\mathcal{A}$, which is neither commutative nor cocommutative.
On the other hand, there are 6 non-semisimple Hopf algebras, 5 of which are pointed (following the notation of \cite{classifying}, these are $\mathcal{A}_2, \mathcal{A}^\prime_4,\mathcal{A}^{\prime \prime}_4,\mathcal{A}^{\prime \prime \prime}_{4,q}$ and $\mathcal{A}_{2,2}$), and only one is not pointed, here denoted by $\mathcal{K}$. The one-dimensional partial actions of the group algebras, duals of finite-dimensional group algebras and 8-dimensional pointed Hopf algebra have been described in the literature \cites{Alvares_Alves_Batista, corresponding}. Hence, the only remaining open cases are the one-dimensional partial actions for the Kac-Paljutkin algebra $\mathcal{A}$ and the non-semisimple non-pointed Hopf algebra $\mathcal{K}$, which are treated in the present paper.

\medbreak

Our second objective concerns partial smash products over the base field. Given a partial action $\lambda:H\to\Bbbk$, the associated partial smash product can be identified with a right coideal subalgebra of $H$. This naturally leads to the converse question of determining which right coideal subalgebras of $H$ can be realized in this way. We investigate this problem beyond the 8-dimensional setting. More precisely, we show that a right coideal subalgebra \(B\) of a finite-dimensional Hopf algebra \(H\) is of the form \(\underline{\Bbbk \# H}\) if and only if the associated right coideal subalgebra \(^{\mathrm{co} B^*} H^*\) of \(H^*\) contains a normalized right integral. This implies in particular that if and \(H\) is cosemisimple, any right coideal subalgebra of \(H\) is obtained from a partial action \(\lambda : H \to \Bbbk\).

\medbreak

This paper is organized as follows.
In Section \ref{sec:preliminaries} we recall some concepts and results concerning partial actions of Hopf algebras (\S \ref{subsec:partial_actions}), and we focus especially on the one-dimensional partial actions in \S \ref{subsec_method}. We also recall the classification of 8-dimensional Hopf algebras and the one-dimensional partial actions known of these algebras in \S \ref{subsec:8dim_Hopf}.
The remaining cases are then the two Hopf algebras $\A$ and $\K$ mentioned above.
In Sections \ref{sec_A} and \ref{sec_K} we compute all one-dimensional partial actions of these Hopf algebras and, by dualizing, we also determine all their partial coactions on the base field.
 Finally, in \S \ref{lambda_Hopf} we investigate when the partial smash product for such partial actions results in a Hopf subalgebra. In \S \ref{right coideal subalgebras}, we address the general problem of determining which right coideal subalgebras of a Hopf algebra can be realized as partial smash products over the base field.

\section{Preliminaries}\label{sec:preliminaries}

In this work, we deal with algebras over a fixed algebraically closed field $\k$ of characteristic zero, although some results hold in positive characteristic.

Unadorned $\otimes$ means $\otimes_{\k}$. We use Sweedler's notation for the coproduct on \(H\): \(\Delta(h) = h_1 \otimes h_2 \in H \otimes H\).
We write $G(H) = \{ g \in H\setminus \{0\} \mid \Delta(g)=g \o g\}$ for the group of the \emph{group-like elements} of a Hopf algebra $H$.
Given $g,h \in G(H)$, an element $x \in H$ is called a \emph{$(g,h)$-primitive element} if $\Delta(x) = x \o g + h \o x$, and the linear space of all $(g,h)$-primitive elements of $H$ is denoted by $P_{g,h}(H)$.

\medbreak

Next, some definitions, results, and examples about partial actions of Hopf algebras will be presented.
For details, we refer to \cites{Alvares_Alves_Batista, enveloping, caenepeel2008partial, guris, corresponding}.

\subsection{Partial actions of Hopf algebras}\label{subsec:partial_actions}

\begin{definition} \label{partial_action}
A \emph{(left) partial action of a Hopf algebra $H$ over an algebra $A$} is a linear map $\cdot : H \otimes A \longrightarrow A,\ h \otimes a \mapsto h \cdot a$, such that:
	\begin{enumerate}[(i)]
		\item $1_H \cdot a=a$;
		\item $h\cdot ab=(h_1\cdot a)(h_2\cdot b)$;
		\item $h\cdot(k\cdot a)=(h_1\cdot 1_A)(h_2k\cdot a)$,
	\end{enumerate}
	for all $h,k\in H$ and $a,b\in A$.
	In this case, $A$ is called a \emph{partial $H$-module algebra}.

    \smallbreak
    
	A (left) partial action is \emph{symmetric} if in addition we have
	\begin{enumerate}[(i)]
        \setcounter{enumi}{3}
		\item $h \cdot ( k \cdot a)=(h_1k \cdot a)(h_2 \cdot 1_A)$,
	\end{enumerate}
	for all $h,k\in H$ and $a\in A$.
\end{definition}

Every global action is a symmetric partial action. Moreover, a partial action is global if and only if $h\cdot 1_A=\varepsilon(h)1_A$, for all $h\in H$, where the map $\varepsilon$ is the counit of $H$.

\medbreak

A one-dimensional partial action of \(H\) is a linear map \(\cdot: H \otimes \Bbbk \to \Bbbk\) satisfying the conditions of Definition \ref{partial_action}. Since \(H \otimes \Bbbk\) is canonically isomorphic to \(H\), such a partial action corresponds to a functional \(\lambda \in H^*\) given by \(\lambda(h) = h~\cdot~1_\Bbbk\). 
The following proposition, which first appeared in \cite{guris}*{Lem. 4.1} and was later particularized for Hopf algebras in \cite{corresponding}*{Prop. 2.1}, determines exactly which functionals are obtained this way.

\begin{prop}\label{k_mod_alg_parc}
	Let $H$ be a Hopf algebra.
    There is a one-to-one correspondence between partial \(H\)-module structures on \(\Bbbk\) and linear maps $\lambda \in H^*$ such that $\lambda(1_H)=1_\k$ and
	\begin{align}\label{eqn_lda}
		\lambda(h)\lambda(k)=\lambda(h_1)\lambda(h_2k),
	\end{align}
	for all $h, k \in H$. 

	Moreover, a partial action is symmetric if and only if the corresponding $\lambda \in H^*$ satisfies the additional condition
	\begin{align}\label{eqn_lda_sim}
\lambda(h)\lambda(k)=\lambda(h_1k)\lambda(h_2)
\end{align}
for all \(h, k \in H\).
\end{prop}

\medbreak

From now on, by a partial action of a Hopf algebra $H$ on its base field $\k$ we mean a linear map $\lambda\in H^*$ satisfying $\lambda(1_H)=1_\k$ and condition \eqref{eqn_lda}.

\medbreak

The next two examples determine all one-dimensional partial actions of group algebras and dual of finite group algebras:
in both cases, the partial actions are parameterized by subgroups.
For more details, see \cite{Alvares_Alves_Batista}*{Sect. 3.2}.

\begin{exa}\label{exemplo_parcialgrupo} Let $G$ be a group. 
The partial actions of $\k G$ on $\k$ are parametrized by the subgroups of $G$.
Precisely, if $N$ is a subgroup of $G$, then  $\lambda_N: \k G \longrightarrow \k$, given by $\lambda_N (g) = \begin{cases}
	1_\k, & \textrm{if } g \in N \\ 0, & \textrm{if } g \notin N\end{cases},$
	for all $g \in G$, is a partial action of $\Bbbk G$ on $\Bbbk$.
Reciprocally, if $\lambda: \Bbbk G \rightarrow \Bbbk$ is a partial action, then there exists a subgroup $N$ of $G$ such that $\lambda = \lambda_N$.
\end{exa}

\begin{exa}\label{exemplo_parcialdualgrupo}
	Let $G$ be a finite group.
	The partial actions of $(\k G)^*$ on $\k$ are also parametrized by the subgroups of $G$.
	Indeed, let $\{g^* \ | \ g \in G\}$ be the dual basis for the canonical basis $\{g \ | \ g \in G\}$ of $\k G$; then for each subgroup $N$ of $G$, the map $\lambda^N : (\k G)^* \longrightarrow \k$ given by
	$\lambda (g^*) = \begin{cases}
	1_\k / |N|, & \textrm{ if } g \in N \\ 0, & \textrm{ if } g \notin N\end{cases},$
	for all $g \in G$, is a partial action of $(\Bbbk G)^\ast$ on $\Bbbk$.
    Conversely, if $\lambda: (\Bbbk G)^\ast \rightarrow \Bbbk$ is a partial action, then there exists a subgroup $N$ of $G$ such that $\lambda = \lambda^N$.
\end{exa}

\subsection{A method to compute one-dimensional partial actions of Hopf algebras}\label{subsec_method}
Here, we briefly recall the method developed in \cite{corresponding}*{Sect. 3.2} for calculating one-dimensional partial actions.

\medbreak

Let $\mathcal{B}$ be a basis of $H$ and consider the linear map $\Lambda: H \longrightarrow \k [ X_b \ | \ b \in \B ]$, $\Lambda(b) = X_b$.
	We define by \emph{partial system associated with $\B$} the following system of equations in the unknowns \(X_b\):
\begin{align}\label{sistema_parcial_associado}
	\left\{ \begin{array}{ll}\Lambda(1_H) = 1_\k \\
		\Lambda(a)\Lambda(b)= \Lambda(a_1)\Lambda(a_2 b)  \end{array}\right. \qquad a, b \in \B.
\end{align}
Hence, the map $\lambda : H \longrightarrow \k$ is a partial action of $H$ on $\k$ if and only if $(\lambda(b))_{b \in \B}$ is a solution of \eqref{sistema_parcial_associado}.
In particular, the global action of $H$ on $\k$, namely $(\varepsilon(b))_{b \in \B}$, is always a solution of the partial system associated with $\B$.

\medbreak

Consider a basis $\B = G(H) \sqcup \B'$ of $H$, where $\sqcup$ denotes the disjoint union.
Then \eqref{sistema_parcial_associado} is rewritten as
\begin{align*}
\left\{\begin{array}{ll} \Lambda(1_H)= 1_\k \\
\Lambda(g)\Lambda(v)= \Lambda(g)\Lambda(g v) \\
\Lambda(u)\Lambda(v)= \Lambda({u}_1)\Lambda({u}_2 v) \end{array}\right. \qquad g \in G(H), u \in \B', v \in \B.
\end{align*}

Let $\lambda:H \longrightarrow \k$ be a partial action.
Then there exists a subgroup $N$ of $G(H)$ such that $\lambda |_{\k G(H)} = \lambda_N$ and so $(\lambda(b))_{b \in \B}$ is a solution of the following system:
\begin{align}\label{sistema_condicao_inicial}
\left\{\begin{array}{ll} \Lambda(g) = 1_\k \\
\Lambda(h) = 0 \\
\Lambda(u) = \Lambda(g u) \\
\Lambda(u)\Lambda(v)= \Lambda({u}_1)\Lambda({u}_2 v) \end{array}\right. \qquad g \in N, h \in G(H) \setminus N, u \in \B', v \in \B.
\end{align}

On the other hand, let $N$ be a subgroup of $G(H)$ and consider a system of equations as \eqref{sistema_condicao_inicial}.
If there exists a solution $(\alpha_b)_{b \in \B}$ to the latter system, then it is also a solution of the partial system associated with $\B$.

\medbreak

For $x, y \in \B'$ and a fixed subgroup $N$ of $G(H)$, we have the following equivalence relation on the set $\B'$:
$ x \sim_N y  \quad \textrm{ if and only if} \quad  \exists  \, \,  g \in N  \textrm{ such that } y = gx.$
Given $ x \in \B'$, we write $[x] = \{y \in \B' \ | \ y \sim_N x \}$ for the \emph{equivalence class of $x$} and $[x] = \{x\} \sqcup [x]^{\perp}$, where $[x]^{\perp} = \{y \in \B' \ | \ y \sim_N x, y \neq x \}$.
Set $\widetilde{N}$ for a \emph{transversal set of the relation $\sim_N$ on $\B'$}. 
Then,  we have a partition of $\B'$ given by $\B' = \widetilde{N} \sqcup N^{\perp}$, where $N^{\perp}=\B' \setminus \widetilde{N} = \bigcup_{x \in \widetilde{N}} [x]^{\perp}$.

Now, consider the following system of equations:
\begin{align}\label{sistema_condicao_inicial_red}
\left\{\begin{array}{ll} \Lambda(g) = 1_\k \\
\Lambda(h) = 0 \\
\Lambda(u) = \Lambda(g u)\end{array}\right.
\qquad g \in N, h \in G(H) \setminus N, u \in \B'.
\end{align}
Then, $(\alpha_b)_{b \in \B}$ is a solution of \eqref{sistema_condicao_inicial_red} if and only if $\alpha_g= 1_\k$, $\alpha_h= 0$ and $\alpha_y=\alpha_x$ for all $g \in N$, $h \in G(H)\setminus N$, $x \in \widetilde{N}$ and $y \in [x]$. Such a solution is completely determined by the free parameters \(\alpha_x\) for $x \in \widetilde{N}$.
Then $(\alpha_b)_{b \in \B}$ is called an \emph{initial $N$-condition for the partial system associated with $\B$}.

Finally, we define the following sets: $\B_{t,s} = \left\{ x \in P_{t,s}(H) \ | \ t \in N,  s \in G(H) \setminus N \right\}$ and
$\widetilde{\B}_N = \widetilde{N} \setminus \left(\bigcup_{t,s \in G(H)} \B_{t,s}\right).$
With these notations, the system
\begin{align}\label{N_reduzido}
\left\{\begin{array}{ll}
\Lambda(u)\Lambda(v)= \Lambda({u}_1)\Lambda({u}_2 v)
\end{array}\right. \qquad u \in \widetilde{\B}_N, v \in \B,
\end{align}
is called of the \emph{$N$-reduced partial system associated with $\B$}.

\begin{thm}[\cite{corresponding}*{Thm. 3.4}]
	Let $\B = G(H) \sqcup \B'$ be a basis of $H$ and $\lambda: H \to \k$ a linear map.
	Then, $\lambda$ is a partial action of $H$ on $\k$ if and only if $N=\{g \in G(H) \ | \ \lambda(g)=1_\k\}$ is a subgroup of $G(H)$ and $(\lambda(b))_{b \in \B} $ is both an initial $N$-condition for the partial system associated with $\B$ and also a solution of the $N$-reduced partial system associated with $\B$.
\end{thm}

\medbreak
	
	Given a partial action $\lambda: H \longrightarrow \k$, we say that $\lambda$ has \emph{initial condition $N$} if $N=\{ g \in G(H) \ | \ \lambda(g)=1_\k \}$.
Thus, we summarize a step-by-step method to compute a partial action of $H$ on $\k$:
\begin{itemize}
	\item[\textbf{Step 1.}] Consider a linear basis $\B=G(H)\sqcup\B^\prime$ of $H$;
	\item[\textbf{Step 2.}] Let $N$ a subgroup of $G(H)$;
	\item[\textbf{Step 3.}] Regard the solutions $(\alpha_b)_{b \in \B}$ of the system \eqref{sistema_condicao_inicial_red};
	\item[\textbf{Step 4.}] Investigate the $N$-reduced partial system associated with $\B$ \eqref{N_reduzido};
	\item[\textbf{Step 5.}] \textbf{Conclusion:} If $(\alpha_b)_{b \in \B}$ is a solution of the $N$-reduced partial system associated with $\B$, then the linear map $\lambda: H \longrightarrow \k$ given by $\lambda(b)=\alpha_b$, for all $b \in \B$, is a partial action of $H$ on $\k$;
	Otherwise, there is not a partial action of $H$ on $\k$ with initial condition $N$.
\end{itemize}
\textbf{Remark:} If one repeats the procedure from step-2 to step-5 for every subgroup of $G(H)$, then all one-dimensional partial actions of $H$ on $\k$ are obtained.

\medbreak

\subsection{The 8-dimensional Hopf algebras and their partial actions on \texorpdfstring{\(\Bbbk\)}{k}}

\label{subsec:8dim_Hopf}
 
\ \newline
 As proved in \cite{stefan}, there are exactly 14 Hopf algebras (up to isomorphism) of  dimension 8.
 Below, we present each of them.

 \begin{itemize}
     \item \textbf{Semisimple 8-dimensional  Hopf algebras:} there are eight (up to isomorphism).    
     Since we have 5 groups of order 8, namely $C_8$, $C_4\times C_2$, $C_2\times C_2\times C_2, \mathfrak{D}_8$ (the dihedral group of order 8) and $\mathcal{Q}$ (the quaternion group), we have:
     \begin{itemize}
         \item 3 commutative and cocommutative Hopf algebras; namely the 3 abelian group algebras $\k C_8$, $\k(C_4\times C_2)$ and $\k(C_2\times C_2\times C_2)$; all self-dual Hopf algebras;
         \item 2 only cocommutative Hopf algebras; namely the 2 non-abelian group algebras $\k \mathfrak{D}_8$ and $\k \mathcal{Q}$;
         \item 2 only commutative Hopf algebras: namely the duals of (non-abelian) group algebras $(\k \mathfrak{D}_8)^*$ and $(\k \mathcal{Q})^*$;
         \item only one non-commutative non-cocommutative Hopf algebra: the Kac-Paljutkin algebra $\A$, that is a self-dual Hopf algebra. 
     \end{itemize}

\smallbreak 

     The one-dimensional partial actions of the group algebras and their duals are known, see Examples \ref{exemplo_parcialgrupo} and \ref{exemplo_parcialdualgrupo}. They are in bijective correspondence with their subgroups.
In Section \ref{sec_A}, the Kac-Paljutkin algebra $\A$ will be presented and its one-dimensional partial actions will be computed.

\medbreak

\item \textbf{Non-semisimple 8-dimensional Hopf algebras}:
there exist six (up to isomorphism); of these, five are pointed and only one is not:
\begin{itemize}
    \item pointed Hopf algebras: $\mathcal{A}_2, \mathcal{A}_4^\prime, \mathcal{A}_4^{\prime \prime}, \mathcal{A}_{4,q}^{\prime \prime \prime}$ and $\mathcal{A}_{2,2}$;
    \item non-pointed Hopf algebra: $\mathcal{K}$.
\end{itemize}

We remark that $\mathcal{A}_{2}$ and $\mathcal{A}_{2,2}$ are self-dual, while $(\mathcal{A}_{4}')^*\cong \mathcal{A}_{4,q}'''$ and $(\mathcal{A}_{4}'')^* \cong \K$.
Moreover, the Hopf algebras $\mathcal{A}_{4,q}'''$ and $\mathcal{A}_{4,-q}'''$ are isomorphic.

\medbreak

For the sake of completeness, we exhibit the Hopf algebra structure of these five non-semisimple pointed Hopf algebras through generators and defining relations. Their partial actions on the base field \(\k\) are summarized in Table \ref{table},
for more details, see \cite{corresponding}.
In Section \ref{sec_K}, the Hopf algebra $\K$ will be presented and its one-dimensional partial actions will be computed.

\smallbreak

$\mathcal{A}_{2} =  \langle g, \, x, \, y  \  |  \ g^2 =1, \,\, x^2=y^2=0, \,\, xg = -gx, \,\, yg=-gy, \,\, yx=-xy \rangle ,$ where $g \in G(\mathcal{A}_{2})$ and $x, y \in P_{1,g}(\mathcal{A}_{2})$;

\smallbreak

$\mathcal{A}_{4}' = \langle g, \, x \  | \ g^4 =1, \,\, x^2=0, \,\, xg = -gx \rangle ,$ where $g \in G(\mathcal{A}_{4}')$ and $x \in P_{1,g}(\mathcal{A}_{4}')$;

\smallbreak

$\mathcal{A}_{4}'' = \langle g, \, x \ | \ g^4 =1, \, x^2=g^2-1, \, xg=-gx\rangle ,$ where $g \in G(\mathcal{A}_{4}'')$ and $x \in P_{1,g}(\mathcal{A}_{4}'')$;

\smallbreak

$\mathcal{A}_{4,q}''' = \langle g, \, x \ | \  g^4 =1, \,\, x^2=0, \,\, gx = qxg \rangle,$ where $q$ is a primitive root of unity of order $4$, $g \in G(\mathcal{A}_{4,q}''')$ and $x \in P_{1,g^2}(\mathcal{A}_{4,q}''')$;

\smallbreak

$\mathcal{A}_{2,2} = \langle g, \, h, \, x \  | \ g^2 = h^2 = 1, \, x^2=0, \, xg = -gx, \, xh = -hx, \, hg=gh \rangle,$ where $g,h \in G(\mathcal{A}_{2,2})$ and $x \in P_{1,g}(\mathcal{A}_{2,2})$.

\smallbreak

  \end{itemize}

\medbreak

In \cite{corresponding}, all one-dimensional partial actions of these Hopf algebras are computed using the method described in \S \ref{subsec_method}.
Similarly to the partial actions of (dual) group algebras on $\Bbbk$, the subgroups of the group-like elements' group play an important role.
Nevertheless, each subgroup sometimes gives rise to a single partial action but sometimes to a family of one or two parameters. 

Next, for each $H \in \{ \mathcal{A}_2, \mathcal{A}_4^\prime, \mathcal{A}_4^{\prime \prime}, \mathcal{A}_{4,q}^{\prime \prime \prime}, \mathcal{A}_{2,2}\}$, we present the (family of) one-dimensional partial actions associated with each subgroup $N$ of $G(H)$. 
If $N=G(H)$, then $\lambda_{G(H)}=\varepsilon_{H}$ (and so it is the global action). All other cases (that is, when $N$ is a proper subgroup) are listed below,  where $\alpha, \beta, \zeta \in \k$ and $\zeta^2=-1$. For details, see \cite{corresponding}*{Prop. 3.7}.

	\begin{center}
		\begin{table}[!ht]
        \begin{tabular}{| c | c | c | c | c | c | c | c | c | }\hline
$\mathcal{A}_2$ & $1$ & $g$ & $x$ & $y$ & $gx$ & $gy$ & $xy$ & $gxy$	\\ \hline
$\lambda_{\{1\}}$ &  $1$ & $0$ & $\alpha$ & $\beta$ & $\alpha$ & $\beta$ & $0$ & $0$ \\ \hline
\end{tabular}
	
           \vspace{0.2cm}
           
\begin{tabular}{| c | c | c | c | c | c | c | c | c | }\hline
$\mathcal{A}_{4}'$	& $1$ & $g$ & $g^2$ & $g^3$ & $x$ & $gx$ & $g^2x$ & $g^3x$	\\ \hline
$\lambda_{\{1, g^2\}}$ &  $1$ & $0$ & $1$ & $0$ & $\alpha$ & $\alpha$ & $\alpha$ & $\alpha$  \\ \hline
$\lambda_{\{1\}}$ &  $1$ & $0$ & $0$ & $0$ & $0$ & $0$ & $0$ & $0$ \\ \hline
\end{tabular}

\vspace{0.2cm}
 
\begin{tabular}{| c | c | c | c | c | c | c | c | c | }\hline
$\mathcal{A}_{4}''$ & $1$ & $g$ & $g^2$ & $g^3$ & $x$ & $gx$ & $g^2x$ & $g^3x$	\\ \hline
$\lambda_{\{1,g^2\}}$ &  $1$ & $0$ & $1$ & $0$ & $\alpha$ & $\alpha$ & $\alpha$ & $\alpha$ \\ \hline
$\lambda_{\{1\}}$ &  $1$ & $0$ & $0$ & $0$ & $\zeta$ & $0$ & $0$ & $\zeta$ \\ \hline
\end{tabular}

  \vspace{0.2cm}              

\begin{tabular}{| c | c | c | c | c | c | c | c | c | }\hline$\mathcal{A}_{4,q}'''$  & $1$ & $g$ & $g^2$ & $g^3$ & $x$ & $gx$ & $g^2x$ & $g^3x$	\\ \hline
$\lambda_{\{1\}}$ &  $1$ & $0$ & $0$ & $0$ & $\alpha$ & $0$ & $\alpha$ & $0$  \\ \hline
$\lambda_{\{1, g^2\}}$ &  $1$ & $0$ & $1$ & $0$ & $0$ & $0$ & $0$ & $0$ \\ \hline 
\end{tabular}

\vspace{0.2cm}           

\begin{tabular}{| c | c | c | c | c | c | c | c | c | }\hline
$\mathcal{A}_{2,2}$	& $1$ & $g$ & $h$ & $gh$ & $x$ & $gx$ & $hx$ & $ghx$	\\ \hline
$\lambda_{\{1\}}$ &  $1$ & $0$ & $0$  & $0$ & $\alpha$ & $\alpha$ & $0$ & $0$  \\ \hline
$\lambda_{\{1, gh\}}$ &  $1$ & $0$ & $0$ & $1$ & $\alpha$ &  $\alpha$ &  $\alpha$ &  $\alpha$ \\ \hline
$\lambda_{\{1, h\}}$ & $ 1$ & $0$ & $1$ & $0$ & $0$ & $0$ & $0$ & $0$  \\ \hline
$\lambda_{\{1, g\}}$ &  $1$ & $1$ & $0$ & $0$ & $0$ & $0$ & $0$ & $0$ \\ \hline    
\end{tabular}
            
\vspace{0.2cm}
            
\caption{One-dimensional partial actions of $\mathcal{A}_2, \mathcal{A}_4^\prime, \mathcal{A}_4^{\prime \prime}, \mathcal{A}_{4,q}^{\prime \prime \prime}$ and $ \mathcal{A}_{2,2}$.}
\label{table}
\end{table}
\end{center}

\subsection{Partial coactions of Hopf algebras}
\label{sec:parco}

    A \textit{(right) partial coaction}  of a Hopf algebra \(H\) on an algebra \(A\) (as introduced in \cite{caenepeel2008partial}, see also \cites{parcorep, BHSV}) is a linear map \(\rho : A \to A \otimes H\) satisfying axioms dual to the ones of Definition \ref{partial_action}: 
    \begin{enumerate}[(i)]
        \item \((A \otimes \varepsilon) \rho (a) = a\);
        \item \(\rho(ab) = \rho(a)\rho(b)\);
        \item \((\rho \otimes H) \rho(a) = (\rho(1_A) \otimes 1_H)\, (A \otimes \Delta)\rho(a)\),
    \end{enumerate}
    for all \(a, b \in A\). It is \textit{symmetric} if moreover
    \begin{enumerate}[(i)]
        \setcounter{enumi}{3}
        \item \((\rho \otimes H) \rho(a) =  (A \otimes \Delta)\rho(a)\, (\rho(1_A) \otimes 1_H)\).
    \end{enumerate}

    \medbreak
    
    Suppose that \(H\) is finite-dimensional and let \(H^*\) be its dual Hopf algebra. A partial \(H^*\)-coaction on \(A\) induces a partial action of \(H\) on \(A\) by setting
    \[h \cdot a = (A \otimes \mathrm{ev}_h) \rho(a)\]
    for all \(h \in H\) and \(a \in A\). For the converse, let \(\B_H = \{h_1, \dots, h_n\}\) be a basis of \(H\) and let \(\{h_1^*, \dots, h_n^*\}\) be the dual basis of \(H^*\). Given a partial \(H\)-action on \(A\), the map
    \[\rho(a) = \sum_{i = 1}^n h_i \cdot a \otimes h_i^*\]
    is a partial \(H^*\)-coaction. This correspondence constitutes an equivalence of categories, hence for finite-dimensional Hopf algebras, the study of partial \(H^*\)-coactions is equivalent to the study of partial \(H\)-actions.
    
     Let \(\lambda : H \to \Bbbk\) be a partial action. The corresponding partial \(H^*\)-coaction is given by 
     \[\rho : \Bbbk \to \Bbbk \otimes H^*,\quad 1 \mapsto 1 \otimes \lambda.\]
    If \(\lambda\) is symmetric, this partial \(H^*\)-coaction can be obtained by the general construction for so-called \textit{partial comodules} presented in \cite{BHSV}. For one-dimensional partial coactions, it goes as follows.

    Let \(B\) be a right coideal subalgebra of \(H^*\) (i.\,e.\ a subalgebra \(B\) such that \(\Delta(B) \subseteq B \otimes H^*\)). Let \(e \in B\) an idempotent which is central in \(B\) with \(\varepsilon(e) = 1\) and such that the ideal generated by \(e\) in \(B\) is one-dimensional. This idempotent then satisfies 
    \begin{equation}
    \label{eq:subcentral}
        ee_1 \otimes e_2 = e_1 e \otimes e_2 = e \otimes e.
    \end{equation}
    Given \(B\), there is at most one such \(e\): it is always a \textit{normalized integral} of \(B\), since \(be = eb = \varepsilon(b) e\) for all \(b \in B\) and \(\varepsilon(e) = 1\).
    Now
    \[\rho_e : \Bbbk \to \Bbbk \otimes H^*, \quad 1 \mapsto 1 \otimes e\]
    is a partial coaction of \(H^*\) on \(\Bbbk\), and every one-dimensional partial \(H^*\)-coaction is of this form. The corresponding partial \(H\)-action is \(\lambda = e \in H^*\). We remark that \eqref{eq:subcentral} is in fact the same as \eqref{eqn_lda} and \eqref{eqn_lda_sim}.

\section{Partial actions and coactions of the algebra \texorpdfstring{$\A$}{A}}\label{sec_A}

 Let $\A$ be the \emph{Kac-Paljutkin algebra}, the unique 8-dimensional semisimple Hopf algebra which is neither commutative nor cocommutative \cite{masuoka}.
As an algebra, \(\A\) is generated by the letters $g, h$ and $x$ subject to the relations
$$g^2=h^2=1, \ \ \ x^2=\tfrac{1}{2}\left(1+g+h-gh\right), \ \ \ hg=gh, \ \ \ xh=gx, \ \ xg=hx.$$

Thus, the set $\B_\A :=\{1,g,h,gh,x,gx,hx,ghx\}$ is a linear basis.
For the coalgebra structure, we have $g, h \in G(\A)$ and $\Delta(x) = 
\tfrac{1}{2}(x\otimes x+x\otimes gx+hx\otimes x-hx\otimes gx)$.
The counit is given by $\varepsilon(k)=1$ for all $k\in \B_\A$ and the antipode by \(S(g) = g\), \(S(h) = h\) and \(S(x) = x\).

\medbreak

For each subgroup $N$ of $G(\mathcal{A})= \{1, g, h, gh \} \cong C_2 \times C_2$, define the linear map $\lambda_N^0: \A \to \Bbbk$ as 
\begin{equation*}
    \lambda_N^0 (t) = \begin{cases}
	1, & \textrm{ if } t \in N \\ 0, & \textrm{ if  } t \in \B_\mathcal{A} \setminus N. \end{cases}
\end{equation*}

Let \(q \in \Bbbk\) a primitive fourth root of unity. 
We also define for $\zeta \in \{ q, -q \}$ the linear map $\lambda_\zeta : \A \to \Bbbk$ as 
\begin{align}\label{lda_zeta}
\begin{split}
\lambda_\zeta(1) = \lambda_\zeta(gh)  = 1 , &  \  \ \ \ 
\lambda_\zeta(g) = \lambda_\zeta(h)  = 0, \\
\lambda_\zeta(x) = \lambda_\zeta(ghx)  = \tfrac{1+\zeta}{2}, & \ \
\lambda_\zeta(gx) = \lambda_\zeta(hx)   = \tfrac{1-\zeta}{2}.
\end{split}
\end{align} 
In the next theorem we show that these linear maps are partial actions of $\mathcal{A}$ on $\Bbbk$ and, together with the counit map, they determine all of them.

\begin{thm}\label{acoes_parciais_A}
    The partial actions of the Hopf algebra \(\A\) on \(\k\) are the counit map \( \varepsilon_\A\) (i.\,e.\ the global action), the linear maps $\lambda_\zeta$ for $ \zeta \in \{q, -q \}$ and $\lambda_N^0$ for each subgroup $N$ of $G(\A)$.
\end{thm}

\begin{proof}
We proceed using the method presented in \S \ref{subsec_method}.
Note that $\B_\A$ is a linear basis of $\A$ obtained extending the group of grouplike elements $G(\A) \cong C_2 \times C_2$, and moreover, since we do not have skew-primitive elements, for any subgroup $N$ of $G(\A)$, the $N$-reduced partial system associated with $\B_\A$ \eqref{N_reduzido} in this case is the following system
\begin{equation}\label{N_reduzido_A}
\left\{\Lambda(u)\Lambda(v) = \Lambda(u_1)\Lambda(u_2 v),
\right. \quad u \in \widetilde{N},\ v \in \B_\A,
\end{equation}
where \(\widetilde{N} = \B_\A \setminus N\).
Hence, to obtain a partial action of $\A$ on $\k$, we only need to consider all the subgroups of $G(\A)$, and for each of these subgroups, analyze the additional relations for the general solution \((\alpha_v)_{v \in \B_\A}\) of \eqref{sistema_condicao_inicial_red} (the initial $N$-condition) to also be a solution of system \eqref{N_reduzido_A} above.
In this case, we obtain a partial action $\lambda: \A \to \k$ given by $\lambda(v)=\alpha_v$, for all $v \in \mathcal{B}_\A$.

\smallbreak

Next, we detail the situation for each subgroup $N$ of $G(\A)$.

\medbreak

\textbf{\underline{Case $N=G(\A)=\{1, g, h, gh\}$:}}

\smallbreak

Here, the initial $N$-condition $(\alpha_v)_{v \in \B_\A}$ is $\alpha_1 = \alpha_g = \alpha_h = \alpha_{gh} = 1$,  and $\alpha_x=\alpha_{gx} = \alpha_{hx} =  \alpha_{ghx} \in \k$ is a free parameter.

On the other hand, system \eqref{N_reduzido_A} becomes 
\begin{equation*}
\left\{
\Lambda(x)\Lambda(v) 
 = \tfrac{1}{2}\left( \Lambda({x})\Lambda({x} v) + \Lambda({x})\Lambda({gx} v) + \Lambda({hx})\Lambda({x} v) - \Lambda({hx})\Lambda({gx} v) \right), 
\right. v \in  \B_\A.
\end{equation*}

Substituting the initial $N$-condition into the system above, all eight rows mean $\alpha_x  = (\alpha_x)^2$.
Thus, the initial $N$-condition $(\alpha_v)_{v \in \B_\A}$ is a solution of the above system if and only if $\alpha_x \in \{0,1\}$.

Therefore, in this case, we obtain two partial actions of $\A$ on $\k$: the global action given by the counit $\varepsilon_\A$ (when $\alpha_x=1)$, given by $\varepsilon_\A(k) = 1,$ for all $k \in \B_\A=\{1, g, h, gh, x, gx, hx, ghx\}$; and the partial action denoted by $\lambda_{G(\A)}^0$ (when $\alpha_x=0$), given by 
\[
 \lambda_{G(\A)}^0 (k)=
     \begin{cases}
         1, &  k \in \{1, g, h, gh\} \\
         0, &  k \in \{x, gx, hx, ghx\}.
     \end{cases}
\]

\medbreak

\textbf{\underline{Case $N=\{1\}$:}}

\smallbreak

In this case, the initial $N$-condition is $\alpha_1 = 1$, $\alpha_g = \alpha_h = \alpha_{gh} = 0$ and $\alpha_x, \alpha_{gx}, \alpha_{hx}, \alpha_{ghx} \in \k$ are four free parameters.

Now, we investigate the additional conditions on $\left(\alpha_v\right)_{v \in \B_\A}$ so that it is also a solution of \eqref{N_reduzido_A}, which in this case is the 32-row system
\begin{equation}\label{N_reduzido_A_1}
\left\{\Lambda(u)\Lambda(v) = \Lambda(u_1)\Lambda(u_2 v),
\right. \quad u \in \{x, gx, hx, ghx\},\ v \in \B_\A.
\end{equation}

First, observe that 
$$\Lambda(x^2)= \Lambda \left(\tfrac{1}{2}\left(1+g+h-gh\right)\right)=\tfrac{1}{2}\left(\Lambda(1)+\Lambda(g)+\Lambda(h)-\Lambda(gh)\right).$$
If we assume that $\left(\alpha_v\right)_{v \in \B_\A}$ is a solution of $\eqref{N_reduzido_A_1}$, then
considering the rows indexed by $(u,v)$ equal to $(x,x)$, $(x,hx)$, $(gx,x)$ and $(gx,hx)$, we obtain respectively
\begin{equation}\label{eq_alpha_1}
    (\alpha_x)^2 = \tfrac{1}{2} \alpha_x, \ \ \alpha_x \alpha_{hx} =  \tfrac{1}{2} \alpha_x, \ \ \alpha_{gx}\alpha_x = \tfrac{1}{2} \alpha_{gx}, \ \  \alpha_{gx} \alpha_{hx} =  \tfrac{1}{2} \alpha_{gx}.
    \end{equation}

Now, the row $(u,v)=(x,1)$ implies
\begin{equation}\label{eq_alpha_2}
\alpha_x = \tfrac{1}{2} \left( (\alpha_x)^2 + \alpha_{x} \alpha_{gx} + \alpha_{hx} \alpha_{x} - \alpha_{hx} \alpha_{gx}\right).
\end{equation}
Then, substituting \eqref{eq_alpha_1} into \eqref{eq_alpha_2}, we obtain $\alpha_x = \tfrac{1}{2} \alpha_x$, that is $\alpha_x =0$.
Hence, \eqref{eq_alpha_1} implies $\alpha_{gx}=0$, and by the rows $(u,v)=(x, gx)$ and $(u,v)=(ghx,1)$ we get $\alpha_{hx}= 0$ and $\alpha_{ghx}=0$, respectively.

\smallbreak

Therefore, if $\left(\alpha_v\right)_{v \in \B_\A}$ is a solution of $\eqref{N_reduzido_A_1}$, then necessarily
$ \alpha_{x} = \alpha_{gx} = \alpha_{hx} = \alpha_{ghx} =0.$

On the other hand, it is very easy to check that such configuration determines a solution to $N$-reduced system \eqref{N_reduzido_A_1}, since substituting $\Lambda(u)$ for $\alpha_u=0$, for each $u \in \{x, gx, hx, ghx\}$, one obtains
\begin{align*}
\left\{
\begin{array}{lll}
\alpha_{x}\Lambda(v)  = \frac{1}{2} \left(\alpha_{x}\Lambda(x v) + \alpha_{x}\Lambda(gx v) + \alpha_{hx}\Lambda(x v) - \alpha_{hx}\Lambda(gx v)\right) \\
\alpha_{gx}\Lambda(v)  = \frac{1}{2} \left(\alpha_{gx}\Lambda(gx v) + \alpha_{gx}\Lambda(x v) + \alpha_{ghx}\Lambda(gx v) - \alpha_{ghx}\Lambda(x v)\right) \\
\alpha_{hx}\Lambda(v)  = \frac{1}{2} \left(\alpha_{hx}\Lambda(hx v) + \alpha_{hx}\Lambda(ghx v) + \alpha_{x}\Lambda(hx v) - \alpha_{x}\Lambda(ghx v)\right) \\
\alpha_{ghx}\Lambda(v)  = \frac{1}{2} \left(\alpha_{ghx}\Lambda(ghx v) + \alpha_{ghx}\Lambda(hx v) + \alpha_{gx}\Lambda(ghx v) - \alpha_{gx}\Lambda(hx v)\right)
\end{array} 
\right.
\end{align*}
for all $v \in \B_\A,$
that is, a system with all 32 rows being $0=0$.

\medbreak

This leads us to conclude that the only partial action in this case is $\lambda_{\{1\}}^0$, given by $\lambda_{\{1\}}^0(1)=1$, and $\lambda_{\{1\}}^0(k)=0$, for $k \in \{g, h, gh, x, gx, hx, ghx\}.$

\medbreak

\textbf{\underline{Case $N=\{1,g\}$:}}

\smallbreak

In this case, the initial $N$-condition is $\alpha_1 = \alpha_g =1$, $\alpha_h = \alpha_{gh} = 0$, and $\alpha_x = \alpha_{gx}, \ \alpha_{hx} =  \alpha_{ghx} \in \Bbbk$ are two free parameters.

On the other hand, system \eqref{N_reduzido_A} becomes the following 16-rows system:
\begin{equation}\label{N_reduzido_A_2}
\left\{\Lambda(u)\Lambda(v) = \Lambda(u_1)\Lambda(u_2 v),
\right. \quad u \in \{x, hx\},\ v \in \B_\A.
\end{equation}

\medbreak

Again, observe that 
$\Lambda(x^2) = \tfrac{1}{2}\left(\Lambda(1)+\Lambda(g)+\Lambda(h)-\Lambda(gh)\right),$
so if we assume that $\left(\alpha_v\right)_{v \in \B_\A}$ is a solution of system \eqref{N_reduzido_A_2}, then considering the rows with index $(u,v)$ equal to $(x,gx)$ and $(hx,g)$ respectively, we obtain $\alpha_x = \alpha_{gx}=0$ and $\alpha_{hx} = \alpha_{ghx} =0$.
Therefore, if $\left(\alpha_v\right)_{v \in \B_\A}$ is a solution of $\eqref{N_reduzido_A_2}$, then necessarily
$ \alpha_{x} = \alpha_{gx} = \alpha_{hx} = \alpha_{ghx} =0.$

\medbreak

Such configuration indeed constitutes a solution to $N$-reduced system \eqref{N_reduzido_A_2}, since substituting $\Lambda(u)$ for $\alpha_u=0$, for each $u \in \{x, gx, hx, ghx\}$, one obtains
\begin{align*}
\left\{
\begin{array}{lll}
\alpha_{x}\Lambda(v)  = \frac{1}{2} \left(\alpha_{x}\Lambda(x v) + \alpha_{x}\Lambda(gx v) + \alpha_{hx}\Lambda(x v) - \alpha_{hx}\Lambda(gx v)\right) \\
\alpha_{hx}\Lambda(v)  = \frac{1}{2} \left(\alpha_{hx}\Lambda(hx v) + \alpha_{hx}\Lambda(ghx v) + \alpha_{x}\Lambda(hx v) - \alpha_{x}\Lambda(ghx v)\right)
\end{array}, 
\right. v \in \B_\A,
\end{align*}
that is, a system with all 16 rows being $0=0$.

Therefore, we conclude that the only partial action in this case is $\lambda_{\{1, g\}}^0$, given by $\lambda_{\{1, g\}}^0(1) = \lambda_{\{1, g\}}^0(g) = 1$ and $\lambda_{\{1, g\}}^0(k)=0$, for $k \in \{h, gh, x, gx, hx, ghx\}.$

\medbreak

\textbf{\underline{Case $N=\{1,h\}$:}}

\smallbreak

This case is analogous to the case $N=\{1, g\}$.
Proceeding with similar computations, one conclude that the only partial action in this case is $\lambda_{\{1, h\}}^0$, given by $\lambda_{\{1, h\}}^0(1) = \lambda_{\{1, h\}}^0(h) = 1$ and $\lambda_{\{1, h\}}^0(k)=0$, for $k \in \{g, gh, x, gx, hx, ghx\}.$

\medbreak

\textbf{\underline{Case $N=\{1,gh\}$:}}

\smallbreak

In this case, the initial $N$-condition is $\alpha_1 = \alpha_{gh} =1$, $\alpha_{g} = \alpha_{h} = 0$, and $\alpha_{x} = \alpha_{ghx}, \ \alpha_{gx} =  \alpha_{hx} \in \Bbbk$ are two free parameters.

On the other hand, system \eqref{N_reduzido_A} means the following 16-rows system:
\begin{equation*}\label{N_reduzido_A_3}
\left\{\Lambda(u)\Lambda(v) = \Lambda(u_1)\Lambda(u_2 v),
\right. \quad u \in \{x, gx\}, v \in \B_\A.
\end{equation*}

Explicitly, the above system is the following: 
\begin{align*}
\left\{
\begin{array}{lll}
\Lambda(x)\Lambda(1)  = \frac{1}{2} \left(\Lambda(x)  \left( \Lambda(x) + \Lambda(gx) \right) + \Lambda(hx) \left( \Lambda(x) - \Lambda(gx) \right) \right) \\
\Lambda(x)\Lambda(g)  = \frac{1}{2} \left(\Lambda(x) \left( \Lambda(hx) + \Lambda(ghx) \right) + \Lambda(hx) \left(\Lambda(hx) - \Lambda(ghx) \right) \right) \\
\Lambda(x)\Lambda(h)  = \frac{1}{2} \left(\Lambda(x) \left(\Lambda(gx) + \Lambda(x) \right) + \Lambda(hx) \left( \Lambda(gx) - \Lambda(x) \right) \right) \\
\Lambda(x)\Lambda(gh)  = \frac{1}{2} \left(\Lambda(x) \left(\Lambda(ghx) + \Lambda(hx) \right) + \Lambda(hx) \left( \Lambda(ghx) - \Lambda(hx) \right) \right) \\
\Lambda(x)\Lambda(x)  = \frac{1}{2} \left(\Lambda(x) \left( \Lambda(1) + \Lambda(g) \right) + \Lambda(hx) \left( \Lambda(h) -\Lambda(gh) \right) \right) \\
\Lambda(x)\Lambda(gx)  = \frac{1}{2} \left(\Lambda(x) \left( \Lambda(h) + \Lambda(gh) \right) + \Lambda(hx) \left( \Lambda(1) - \Lambda(g) \right) \right) \\
\Lambda(x)\Lambda(hx)  = \frac{1}{2} \left(\Lambda(x) \left( \Lambda(1) + \Lambda(g) \right) + \Lambda(hx) \left( \Lambda(gh) - \Lambda(h) \right) \right) \\
\Lambda(x)\Lambda(ghx)  = \frac{1}{2} \left(\Lambda(x) \left( \Lambda(h) + \Lambda(gh) \right) + \Lambda(hx) \left( \Lambda(g) - \Lambda(1)\right) \right) \\
\Lambda(gx)\Lambda(1)  = \frac{1}{2} \left(\Lambda(gx) \left( \Lambda(gx) + \Lambda(x) \right) + \Lambda(ghx) \left( \Lambda(gx) - \Lambda(x) \right) \right)\\
\Lambda(gx)\Lambda(g)  = \frac{1}{2} \left(\Lambda(gx)\left(\Lambda(ghx) + \Lambda(hx) \right) + \Lambda(ghx) \left( \Lambda(ghx) -\Lambda(hx)\right) \right) \\
\Lambda(gx)\Lambda(h)  = \frac{1}{2} \left(\Lambda(gx) \left(\Lambda(x) + \Lambda(gx) \right) + \Lambda(ghx) \left( \Lambda(x) - \Lambda(gx) \right) \right)\\
\Lambda(gx)\Lambda(gh)  = \frac{1}{2} \left(\Lambda(gx) \left( \Lambda(hx) + \Lambda(ghx) \right) + \Lambda(ghx) \left( \Lambda(hx) - \Lambda(ghx)\right) \right) \\
\Lambda(gx)\Lambda(x)  = \frac{1}{2} \left(\Lambda(gx) \left( \Lambda(1) + \Lambda(g) \right) + \Lambda(ghx) \left( \Lambda(gh) - \Lambda(h) \right)\right)\\
\Lambda(gx)\Lambda(gx)  = \frac{1}{2} \left(\Lambda(gx) \left( \Lambda(h) + \Lambda(gh) \right) + \Lambda(ghx) \left( \Lambda(g) - \Lambda(1)\right) \right)\\
\Lambda(gx)\Lambda(hx)  = \frac{1}{2} \left(\Lambda(gx) \left( \Lambda(1) + \Lambda(g) \right) + \Lambda(ghx) \left( \Lambda(h) - \Lambda(gh)\right) \right) \\
\Lambda(gx)\Lambda(ghx)  = \frac{1}{2} \left(\Lambda(gx) \left( \Lambda(h) + \Lambda(gh) \right) + \Lambda(ghx) \left( \Lambda(1) - \Lambda(g) \right) \right)
\end{array} 
\right.
\end{align*}

Hence, substituting the initial $N$-condition $\left(\alpha_v\right)_{v \in \B_\A}$ into the above system, we obtain the following set of six relations:
\begin{align}\label{relacoes_A}
\left\{
\begin{array}{lll}
0  = (\alpha_{x})^2 + (\alpha_{gx})^2 \\
\alpha_{x}  = \frac{1}{2} \left( (\alpha_{x})^2  + 2 \ \alpha_{x} \ \alpha_{gx}  - (\alpha_{gx})^2 \right) \\
\alpha_{gx}  = \frac{1}{2} \left(  (\alpha_{gx})^2 + 2 \ \alpha_{x} \ \alpha_{gx}  - (\alpha_{x})^2 \right)\\
\alpha_{x} \ \alpha_{gx}  = \frac{1}{2} \left(\alpha_{x} + \alpha_{gx} \right) \\
(\alpha_{x})^2  = \frac{1}{2} \left(\alpha_{x}  - \alpha_{gx} \right) \\
(\alpha_{gx})^2  = \frac{1}{2} \left(\alpha_{gx} - \alpha_{x}  \right)
\end{array} 
\right.
\end{align}

System \eqref{relacoes_A} has exactly three solutions: 
$$(\alpha_x, \alpha_{gx}) = (0,0), \ (\alpha_x, \alpha_{gx}) = \left( \tfrac{1+q}{2}, \tfrac{1-q}{2} \right) \textrm{ and } (\alpha_x, \alpha_{gx}) = \left(\tfrac{1-q}{2},\tfrac{1+q}{2} \right),$$
that lead us respectively to the partial actions $\lambda_{\{1, gh\}}^0$, $\lambda_q$ and $\lambda_{-q}$ (as defined in \eqref{lda_zeta}), and it concludes the proof. \qedhere
\end{proof}

\begin{prop}
    All partial actions of $\A$ on $\Bbbk$ determined in Theorem \ref{acoes_parciais_A} are symmetric.
\end{prop}

\begin{proof}
To verify that a partial action of $\A$ on $\Bbbk$ is symmetric, we check if it satisfies symmetric condition \eqref{eqn_lda_sim}.
The global action (that is, the counit map) is always symmetric, and for each subgroup $N$ of $G(\A)$, it is easy to check that the partial action $\lambda_N^0$ satisfies \eqref{eqn_lda_sim}.
Thus, it remains to consider partial actions $\lambda_\zeta$, where $\zeta \in \{q, -q\}$. Recall that $\lambda_\zeta$ is defined as \eqref{lda_zeta}.

We need to check if 
\begin{equation}\label{eq_sim_A}
\lambda_\zeta(u)\lambda_\zeta(v) = \lambda_\zeta(u_1 v) \lambda_\zeta(u_2)
\end{equation}
holds for all $u, v \in \B_\A.$
This clearly holds for all $u \in \{1, g, h, gh\}$ and $v \in \B_\A$, because \(\lambda_\zeta(g) = \lambda_\zeta(h) = 0\) and \(\lambda_\zeta(ghv) = \lambda_\zeta(v)\). 

For $u$ equal to $x, gx, hx$ and $ghx$, and $v \in \B_\A$, \eqref{eq_sim_A} becomes respectively:
\begin{align*}
\lambda_\zeta(x)\lambda_\zeta(v) & = \tfrac{1}{2} \left( \lambda_\zeta(x v) \left( \lambda_\zeta(x) + \lambda_\zeta(gx) \right) + \lambda_\zeta(hx v) \left( \lambda_\zeta(x) - \lambda_\zeta(gx) \right) \right),\\
\lambda_\zeta(gx)\lambda_\zeta(v) & = \tfrac{1}{2} \left( \lambda_\zeta(gx v) \left( \lambda_\zeta(gx) + \lambda_\zeta(x) \right) + \lambda_\zeta(ghx v) \left( \lambda_\zeta(gx) - \lambda_\zeta(x) \right) \right),\\
\lambda_\zeta(hx)\lambda_\zeta(v) & = \tfrac{1}{2} \left( \lambda_\zeta(hx v) \left( \lambda_\zeta(hx) + \lambda_\zeta(ghx) \right) + \lambda_\zeta(x v) \left( \lambda_\zeta(hx) - \lambda_\zeta(ghx) \right) \right),\\
\lambda_\zeta(ghx)\lambda_\zeta(v) & = \tfrac{1}{2} \left( \lambda_\zeta(ghx v) \left( \lambda_\zeta(ghx) + \lambda_\zeta(hx) \right) + \lambda_\zeta(gx v) \left( \lambda_\zeta(ghx) - \lambda_\zeta(hx) \right) \right).
\end{align*}
Since $\lambda_\zeta(x)+\lambda_\zeta(gx) =\lambda_\zeta(ghx)+\lambda_\zeta(hx) =1$ and $\lambda_\zeta(x)-\lambda_\zeta(gx) = \lambda_\zeta(ghx)-\lambda_\zeta(hx) =\zeta$, the above equations are equivalent to
\begin{align}
\lambda_\zeta(x)\lambda_\zeta(v) & = \tfrac{1}{2} \left( \lambda_\zeta(x v) + \zeta \, \lambda_\zeta(hx v)  \right), \label{lambda_zeta_xv}\\
\lambda_\zeta(gx)\lambda_\zeta(v) & = \tfrac{1}{2} \left( \lambda_\zeta(gx v) - \zeta \, \lambda_\zeta(ghx v) \right), \label{lambda_zeta_gxv}\\
\lambda_\zeta(hx)\lambda_\zeta(v) & = \tfrac{1}{2} \left( \lambda_\zeta(hx v) - \zeta \, \lambda_\zeta(x v) \right), \label{lambda_zeta_hxv}\\
\lambda_\zeta(ghx)\lambda_\zeta(v) & = \tfrac{1}{2} \left( \lambda_\zeta(ghx v) + \zeta \, \lambda_\zeta(gx v)  \right).\label{lambda_zeta_ghxv}
\end{align}
Moreover, as $\lambda_\zeta(v)=\lambda_\zeta(ghv)$ for all $v \in \B_\A$, 
\eqref{lambda_zeta_ghxv} $\Leftrightarrow$ \eqref{lambda_zeta_xv} and \eqref{lambda_zeta_hxv} $\Leftrightarrow$ \eqref{lambda_zeta_gxv}.
Thus, we only need to check if
\eqref{lambda_zeta_xv} and \eqref{lambda_zeta_gxv}
hold for all $v \in \B_\A$.

Since $\lambda_\zeta(x^2)= \lambda_\zeta(ghx^2)=0,\ \lambda_\zeta(gx^2)= \lambda_\zeta(hx^2)= 1$ and $\lambda_\zeta(v)=\lambda_\zeta(ghv)$ for all $v \in \B_\A$, developing expressions \eqref{lambda_zeta_xv} and \eqref{lambda_zeta_gxv}, we obtain that all sixteen equalities hold, since the following equalities are valid:
\begin{gather*}
\lambda_\zeta(x) = \tfrac{1}{2} \left( \lambda_\zeta (x) + \zeta \, \lambda_\zeta(gx) \right), \\
\lambda_\zeta(gx) = \tfrac{1}{2} \left( \lambda_\zeta (gx) - \zeta \, \lambda_\zeta(x) \right), \\
\lambda_\zeta (gx) + \zeta \, \lambda_\zeta(x)  =  \lambda_\zeta (x) - \zeta \, \lambda_\zeta(gx) = 0, \\
    \lambda_\zeta (x)^2 = - \, \lambda_\zeta (gx)^2 = \tfrac{\zeta}{2}, \\
    \lambda_\zeta (x) \lambda_\zeta(gx) = \tfrac{1}{2}.
\end{gather*}

Therefore, for $\zeta = \pm q$, the partial action $\lambda_\zeta$ is a symmetric partial action.
\end{proof}

\medbreak

The simple \textit{partial comodules} of the Kac-Paljutkin algebra were classified in \cite{BHSV}*{Sect. 5}.
    Since \(\A\) is self-dual, every one-dimensional partial coaction of \(\A\) corresponds to one of its one-dimensional (symmetric) partial actions, as explained in \S \ref{sec:parco}.
    Let us describe this correspondence explicitly. Let \(\xi\) be a primitive eighth root of unity in \(\Bbbk\) and write \(q := \xi^2\), \(\sqrt{2} := \frac{2\xi}{1 + \xi^2}\).

    Consider again the linear basis \(\B_\A = \{1, g, h, gh, x, gx, hx, ghx\}\) of \(\A\), and its dual basis \(\B_\A^*\) of \(\A^*\). A Hopf algebra isomorphism \(\psi : \A \to \A^*\) is given by
    \begin{align*}
	1 &\mapsto 1^* + g^* + h^* + gh^* + x^* + gx^* + hx^* + ghx^* \\
	g &\mapsto 1^* - g^* - h^* + gh^* - q x^* +q gx^* + qhx^* -q ghx^* \\
	h &\mapsto 1^* - g^* - h^* + gh^* + qx^* -q gx^* -q hx^* + qghx^* \\
	gh &\mapsto 1^* + g^* + h^* + gh^* - x^* - gx^* - hx^* - ghx^* \\
	x &\mapsto 1^* - qg^* +q h^* - gh^* + \sqrt{2} x^* - \sqrt{2} ghx^* \\
	gx &\mapsto 1^* + q g^* - qh^* - gh^*  + q \sqrt{2} gx^* -q \sqrt{2} hx^*  \\
	hx &\mapsto 1^* +q g^* -q h^* - gh^* -q \sqrt{2} gx^* + q \sqrt{2} hx^* \\
	ghx &\mapsto 1^* - qg^* +q h^* - gh^* - \sqrt{2} x^* + \sqrt{2} ghx^*.
\end{align*}

There are exactly 8 right coideal subalgebras of \(\A\) (see \cite{DT}*{Sect. 3.2}): apart from the group algebras \(\langle 1 \rangle\), \(\langle 1, g \rangle\), \(\langle 1, h \rangle\), \(\langle 1, gh \rangle\), \(\langle 1, g, h, gh \rangle\) and \(\A\) itself, there are the two four-dimensional right coideal subalgebras
\begin{align*}
	S_1 &= \left\langle 1, gh, \frac{1-q}{2}x + \frac{1 + q}{2}gx, \frac{1 + q}{2}hx + \frac{1 - q}{2}ghx \right\rangle = \langle 1, gh, s, ghs \rangle, \\
	S_2 &= \left\langle 1, gh, \frac{1 + q}{2}x + \frac{1 - q}{2}gx, \frac{1 - q}{2}hx + \frac{1 + q}{2}ghx \right\rangle = \langle 1, gh, \bar{s}, gh \bar{s} \rangle,
\end{align*}
where \(s := \frac{1 - q}{2}x + \frac{1 + q}{2}gx\) and \(\bar{s} := \frac{1 + q}{2}x + \frac{1 - q}{2}gx\).

\smallbreak

Each of these right coideal subalgebras \(S\) contains a unique normalized integral \(e_S\), which are listed in \cite{BHSV}*{Table 2}. The corresponding partial \(\A\)-action is then \(\psi(e) \in \A^*\). 
Concretely, we obtain
\begin{itemize}
    \item for \(S = \langle 1 \rangle\): \(e_S = 1\), and \(\psi(e_S) = \varepsilon\);
    \item for \(S = \langle 1, g \rangle\): \(e_S = \frac{1 + g}{2}\), and \(\psi(e_S) = \lambda_{-q}\);
    \item for \(S = \langle 1, h \rangle\): \(e_S = \frac{1 + h}{2}\), and \(\psi(e_S) = \lambda_{q}\);
    \item for \(S = \langle 1, gh \rangle\): \(e_S = \frac{1 + gh}{2}\), and \(\psi(e_S) = \lambda^0_{\{1, g, h, gh\}}\);
    \item for \(S = \langle 1, g, h, gh \rangle\): \(e_S = \frac{1 + g + h+ gh}{4}\), and \(\psi(e_S) = \lambda^0_{\{1, gh\}}\);
    \item for \(S = S_1\): \(e_S = \frac{1 + gh + s+ ghs}{4}\), and \(\psi(e_S) = \lambda^0_{\{1, h\}}\);
    \item for \(S = S_2\): \(e_S = \frac{1 + gh + \bar{s}+ gh\bar{s}}{4}\), and \(\psi(e_S) = \lambda^0_{\{1, g\}}\);
    \item for \(S = \A\): \(e_S = \frac{1 + g + h+ gh + x + gx + hx + ghx}{8}\), and \(\psi(e_S) = \lambda^0_{\{1\}}\).
\end{itemize}

\section{Partial actions and coactions of the algebra \texorpdfstring{$\K$}{K}}\label{sec_K}

 Let \(\xi\) be a primitive eighth root of unity in \(\Bbbk\) and write \(q := \xi^2\), \(\sqrt{2} := \frac{2\xi}{1 + \xi^2}\).
As an algebra, $\K$ is generated by the letters $a, b, c$ and $d$ subject to the relations
\begin{eqnarray*}
a^4 = ad = da = 1, \ \ \ b^2 = c^2 = bc = cb = 0,  \ \ \  a^2 c = b, \\ ab = q \ ba, \ \ \ ac = q \ ca, \ \ \ bd = q \ db,  \ \ \  cd = q \ dc.
\end{eqnarray*}

Notice that the above relations also imply the following ones:
\begin{equation*}
a^3 = d,\ a^2b = c,\ a^3b = ac,\ a^3c = ab,\ d^2 = a^2,\ a^2d = a,\ a^3d = a^2,\ db = ac,\ dc = ab.
\end{equation*}
Thus, a linear basis of $\K$ is $\B_\K = \{1, a, b, c, d, a^2, ab, ac \}$.

\medbreak

The coproduct \(\Delta: \K  \to \K \otimes \K\) is determined by
\begin{eqnarray*}
& \Delta(1)  = 1 \otimes 1,  & \Delta(a^2) = a^2 \otimes a^2, \\ 
& \Delta(a)  = a \otimes a + b \otimes c,  & \Delta(b) = b \otimes d + a \otimes b, \\
& \Delta(c)   = c \otimes a + d \otimes c, &\Delta(d)  = c \otimes b + d \otimes d, \\
& \Delta(ab)  = ab \otimes 1 + a^2 \otimes ab,  & \Delta(ac) = ac \otimes a^2 + 1 \otimes ac, 
\end{eqnarray*}
and the counit \(\varepsilon: \K \to \Bbbk\) is given by
$$
\varepsilon(1) = \varepsilon(a) = \varepsilon(a^2) = \varepsilon(d) = 1, \ \ \ 
\varepsilon(b) = \varepsilon(c) = \varepsilon(ab) = \varepsilon(ac) = 0.
$$
The antipode satisfies \(S(a) = d,\ S(b) = qb,\ S(c) = -qc,\ S(d) = a\).

Note that $a^2$ is a grouplike element and that $ab$ and $ac$ are skew-primitive elements, that is, $G(\K) = \{ 1, a^2 \} = C_2$, $ab \in P_{1, a^2}(\K)$ and $ac \in P_{a^2, 1}(\K)$.

\medbreak

In order to obtain explicitly all the partial actions of $\mathcal{K}$ on $\Bbbk$, we define the following linear maps:

\begin{itemize}
    \item the map $\lambda_{\{1,a^2\}}^0: \K \to \Bbbk$ as 
    \begin{equation}\label{lda_N}
    \lambda_{\{1,a^2\}}^0 (t) = \begin{cases}
	1, & \textrm{ if } t \in \{1,a^2\} \\ 0, & \textrm{ otherwise }\end{cases}, \textrm{ for all $t \in \B_\mathcal{K}$;} 
\end{equation}

\item for each $\alpha \in \Bbbk$, the map $\lambda_\alpha: \K \to \Bbbk$ as 
\begin{equation}\label{lda_alpha} 
\begin{split}\lambda_\alpha(1) = 1, \qquad \lambda_\alpha(ab) = \lambda_\alpha(ac) = \alpha, \\ 
\lambda_\alpha(a^2) = \lambda_\alpha(a) = \lambda_\alpha(b) = \lambda_\alpha(c) = \lambda_\alpha(d) = 0;
\end{split}
\end{equation}

\item the map $\lambda_+ : \K \to \Bbbk$ as 
\begin{equation}\label{lda_plus} 
\begin{split}
\lambda_+(1) = 1, \ \ \lambda_+(a^2) = 0, \ \
\lambda_+(a) = \tfrac{1 - q}{2}, \ \ \lambda_+(d) = \tfrac{1 + q}{2}, \\
\lambda_+(b) = \lambda_+(c) =\lambda_+(ab) = \lambda_+(ac) = \tfrac{\sqrt{2}}{2},
\end{split}
\end{equation}
and the map $\lambda_{-} : \K \to \Bbbk$ as 
\begin{equation}\label{lda_minus} 
\begin{split}
\lambda_{-}(1) = 1, \ \ \lambda_{-}(a^2) = 0, \ \
\lambda_{-}(a) = \tfrac{1 - q}{2}, \ \ \lambda_{-}(d) = \tfrac{1 + q}{2}, \\
\lambda_{-}(b) = \lambda_{-}(c) =\lambda_{-}(ab) = \lambda_{-}(ac) = -\tfrac{\sqrt{2}}{2}.
\end{split}
\end{equation}
\end{itemize}

\begin{thm}\label{acoes_parciais_K}
 The partial actions of the Hopf algebra \(\K\) on \(\k\) are the counit map \( \varepsilon_\K\) (the global action), and the linear maps $\lambda_{\{1, a^2\}}^0$, $\lambda_{+}$, $\lambda_{-}$ and $\lambda_\alpha$, for all $\alpha \in \k.$
\end{thm}

\begin{proof}
As in the proof of Theorem \ref{acoes_parciais_A}, we proceed using the method outlined in \S \ref{subsec_method}.
Since $G(\K) = C_2$, we only have two possibilities for the subgroup $N$: $\{1\}$ or $C_2= \{1, a^2\}$.

\medbreak

\textbf{\underline{Case $N=G(\K)=\{1, a^2\}$:}}

\smallbreak

Recall that $a^2a = d,\ a^2b=c$ and $a^2ab=ac$, then the initial $N$-condition $(\alpha_v)_{v \in \B_\K}$ is $\alpha_1 = \alpha_{a^2} = 1$, and $\alpha_a = \alpha_{d},\ \alpha_{b} =  \alpha_{c},\ \alpha_{ab} = \alpha_{ac} \in \k$ are three free parameters.

On the other hand, the $N$-reduced partial system associated with $\B_\K$ \eqref{N_reduzido} means 
\begin{equation*}
\left\{
\Lambda(u)\Lambda(v)
 = \Lambda(u_1)\Lambda(u_2 v), 
\right. \ \ \ u \in \tilde{\B}_N = \{a, b, ab\}, \ v \in \B_\K.
\end{equation*}

Now, in order to obtain the additional relations so that the initial $N$-condition $(\alpha_v)_{v \in \B_\K}$ is also a solution of the 24-rows system above, we substitute $(\alpha_v)_{v \in \B_\K}$ into it and obtain the following relations (discounting the repeated rows):
\begin{align*}
\left\{
\begin{array}{lll}
\alpha_a  = (\alpha_{a})^2 + (\alpha_{b})^2 \\
\alpha_a  = (\alpha_{a})^2 - (\alpha_{b})^2 \\
(\alpha_{a})^2 = \alpha_a + q^{-1} \ \alpha_{b} \alpha_{ab} \\
(\alpha_{a})^2 = \alpha_a - q^{-1} \ \alpha_{b} \alpha_{ab} \\
\alpha_a \ \alpha_{b} = \alpha_a \ \alpha_{ab} \\
\alpha_b = 2 \ \alpha_a \ \alpha_b
\end{array} 
\right.
& \qquad \qquad \left\{
\begin{array}{lll}
\alpha_b = 0 \\
\alpha_a \ \alpha_b = q^{-1} \ \alpha_{ab} \ \alpha_a + \alpha_b \\
(\alpha_b)^2 = \alpha_{ab} \ \alpha_b \\
\alpha_a \ \alpha_b = q \ \alpha_a  \ \alpha_{ab} + \alpha_b \\
\alpha_b \ \alpha_{ab} = (\alpha_b)^2 \\
\alpha_{ab} = 0
\end{array} 
\right. 
\end{align*}
which is equivalent to
\begin{align*}
\left\{
\begin{array}{lll}
\alpha_a  = (\alpha_{a})^2 \\
\alpha_b = 0 \\
\alpha_{ab} = 0 .
\end{array} 
\right.
\end{align*}

Therefore, in this case, we obtain two partial actions of $\K$ on $\k$: the global action given by the counit $\varepsilon_\K$ (when $\alpha_a=1)$, and the partial action denoted by $\lambda_{\{1,a^2\}}^0$ (when $\alpha_a=0$). See \eqref{lda_N}.

\medbreak

\textbf{\underline{Case $N=\{1\}$:}}

\smallbreak

In this case the initial $N$-condition $(\alpha_v)_{v \in \B_\K}$ is $\alpha_1 = 1,$ $\alpha_{a^2} = 0$ and $\alpha_a,$ $ \alpha_b,$ $\alpha_c,$ $\alpha_{d},$ $\alpha_{ab},$ $\alpha_{ac} \in \k$ are six free parameters, and the $N$-reduced partial system associated with $\B_\K$ \eqref{N_reduzido} means 
\begin{equation*}
\left\{
\Lambda(u)\Lambda(v)
 = \Lambda(u_1)\Lambda(u_2 v), 
\right. \ \ \ u \in \tilde{\B}_N = \{a, b, c, d, ac\}, \ v \in \B_\K.
\end{equation*}

Now, in order to obtain the additional relations so that the initial $N$-condition $(\alpha_v)_{v \in \B_\K}$ is also a solution of the 40-rows system above, we substitute $(\alpha_v)_{v \in \B_\K}$ into it and obtain the following 40 relations:
\begin{align}\label{sist_K}
\left\{
\begin{array}{lll}
\alpha_a  = (\alpha_{a})^2 + \alpha_{b} \ \alpha_{c} \\
(\alpha_{b})^2 = \alpha_{a} \ \alpha_{d} \\
(\alpha_{a})^2 = q^{-1} \ \alpha_b  \ \alpha_{ac}\\
\alpha_{a} \ \alpha_{b} = \alpha_{a} \alpha_{ab} \\
\alpha_a \ \alpha_{c} = \alpha_a \ \alpha_{ac} \\
\alpha_a \ \alpha_{d} = \alpha_{a} - q^{-1} \ \ \alpha_{b} \ \alpha_{ab} \\
\alpha_a \ \alpha_{ac} = \alpha_{a} \ \alpha_{b} \\
\alpha_a \ \alpha_{ab} = \alpha_{a} \ \alpha_c \\
\alpha_{b} = \alpha_{a} \ \alpha_{b} + \alpha_{b} \ \alpha_{d} \\
\alpha_{a} \ \alpha_{b} = \alpha_{a} \ \alpha_{c}\\
\alpha_{a} \ \alpha_{b} = q^{-1} \ \alpha_{a} \ \alpha_{ab} + \alpha_{b} \\
(\alpha_{b})^2 = \alpha_{b} \ \alpha_{ac} \\
\alpha_{b} \ \alpha_{c} = \alpha_{b} \ \alpha_{ab} \\
\alpha_{b} \ \alpha_{d} = q \ \alpha_{a} \ \alpha_{ac} \\
(\alpha_{b})^2 = \alpha_{b} \ \alpha_{ab} \\
\alpha_{b} \ \alpha_{ac} = \alpha_{b} \ \alpha_{c}\\
\alpha_{c} = \alpha_{a} \ \alpha_{c} + \alpha_{c} \ \alpha_{d} \\
\alpha_{b} \ \alpha_{d} = \alpha_{c} \ \alpha_{d} \\
\alpha_{a} \ \alpha_{c} = q^{-1} \ \alpha_{d} \ \alpha_{ac} \\
\alpha_{b} \ \alpha_{c} = \alpha_{c} \ \alpha_{ab}
\end{array} 
\right. & \qquad \qquad
\left\{
\begin{array}{lll}
(\alpha_{c})^2 = \alpha_{c} \ \alpha_{ac} \\
\alpha_{c} \ \alpha_{d} = \alpha_{c} - q^{-1} \ \alpha_{d} \ \alpha_{ab} \\
(\alpha_{c})^2 = \alpha_{c} \ \alpha_{ab} \\
\alpha_{c} \ \alpha_{ac} = \alpha_{b} \ \alpha_{c} \\
\alpha_{d} = \alpha_{b} \ \alpha_{c} + (\alpha_{d})^2 \\
(\alpha_{c})^2 = \alpha_{a} \ \alpha_{d} \\
\alpha_{a} \ \alpha_{d} = q^{-1} \ \alpha_{c} \ \alpha_{ab} + \alpha_{d} \\
\alpha_{b} \ \alpha_{d} = \alpha_{d} \ \alpha_{ac} \\
\alpha_{c} \ \alpha_{d} = \alpha_{d} \ \alpha_{ab} \\
(\alpha_{d})^2 = q \ \alpha_{c} \ \alpha_{ac} \\
\alpha_{d} \ \alpha_{ab} = \alpha_{b} \ \alpha_{d} \\
\alpha_{d} \ \alpha_{ac} = \alpha_{c} \ \alpha_{d} \\
\alpha_{ac} = \alpha_{ac} \\
\alpha_{ac} = \alpha_{ab} \\
\alpha_{a} \ \alpha_{ac} = \alpha_{d} \ \alpha_{ac} - q \ \alpha_{b} \\
\alpha_{b} \ \alpha_{ac} = \alpha_{c} \ \alpha_{ac} \\
\alpha_{c} \ \alpha_{ac} = \alpha_{b} \ \alpha_{ac} \\
\alpha_{d} \ \alpha_{ac} = \alpha_{a} \ \alpha_{ac} + q \ \alpha_{c} \\
\alpha_{ab} \ \alpha_{ac} = (\alpha_{ac})^2 \\
(\alpha_{ac})^2 = \alpha_{ab} \ \alpha_{ac}
\end{array} 
\right. 
\end{align}
which in turn is equivalent to following 28 relations
\begin{align}\label{relacoes_K_red}
\left\{
\begin{array}{lll}
\alpha_{ac} = \alpha_{ab} \\
\alpha_{a} = (\alpha_{a})^2 + \alpha_{b} \ \alpha_{c} \\
\alpha_{a} = \alpha_{a} \ \alpha_{d} + q^{-1} \ \alpha_{b} \ \alpha_{ab} \\
\alpha_{b} = \alpha_{a} \ \alpha_{b} + \alpha_{b} \ \alpha_{d} \\
\alpha_{b} = \alpha_{a} \ \alpha_{b} - q^{-1} \ \alpha_{a} \ \alpha_{ab} \\
\alpha_{b} = q \ \alpha_{a} \ \alpha_{ab} - q \ \alpha_{d} \ \alpha_{ab} \\
\alpha_{c} = \alpha_{a} \ \alpha_{c} + \alpha_{c} \ \alpha_{d} \\
\alpha_{c} = \alpha_{c} \ \alpha_{d} + q^{-1} \ \alpha_{d} \ \alpha_{ab} \\
\alpha_{c} = q \ \alpha_{a} \ \alpha_{ab} - q \ \alpha_{d} \ \alpha_{ab} \\
\alpha_{d} = \alpha_{b} \ \alpha_{c} + (\alpha_{d})^2 \\
\alpha_{d} = \alpha_{a} \ \alpha_{d} - q^{-1} \ \alpha_{c} \ \alpha_{ab} \\
(\alpha_{a})^2 = q^{-1} \ \alpha_{b} \ \alpha_{ab} \\
\alpha_{a} \ \alpha_{b} = \alpha_{a} \ \alpha_{ab} \\
\alpha_{a} \ \alpha_{b} = \alpha_{a} \ \alpha_{c} \\
\end{array} 
\right.
& \qquad \qquad 
\left\{
\begin{array}{lll}
\alpha_{a} \ \alpha_{c} = \alpha_{a} \ \alpha_{ab} \\
\alpha_{a} \ \alpha_{c} = q^{-1} \ \alpha_{d} \ \alpha_{ab} \\
\alpha_{a} \ \alpha_{d} = (\alpha_{b})^2\\
\alpha_{a} \ \alpha_{d} = (\alpha_{c})^2\\
\alpha_{a} \ \alpha_{ab} = q^{-1} \ \alpha_{b} \ \alpha_{d}\\
(\alpha_{b})^2 = \alpha_{b} \ \alpha_{ab} \\
\alpha_{b} \ \alpha_{c} = \alpha_{b} \ \alpha_{ab} \\
\alpha_{b} \ \alpha_{c} = \alpha_{c} \ \alpha_{ab} \\
\alpha_{b} \ \alpha_{d} = \alpha_{c} \ \alpha_{d} \\
\alpha_{b} \ \alpha_{d} = \alpha_{d} \ \alpha_{ab} \\
\alpha_{b} \ \alpha_{ab} = \alpha_{c} \ \alpha_{ab} \\
(\alpha_{c})^2 = \alpha_{c} \ \alpha_{ab} \\
\alpha_{c} \ \alpha_{d} = \alpha_{d} \ \alpha_{ab} \\
(\alpha_{d})^2 = q \ \alpha_{c} \ \alpha_{ab} \\
\end{array} 
\right.
\end{align}

\medbreak

Now, we investigate the relations \eqref{relacoes_K_red}.

\begin{itemize}
    \item If $\alpha_a=0$, then \eqref{relacoes_K_red} means
\begin{align*}
\left\{
\begin{array}{lll}
\alpha_{ac} = \alpha_{ab} \\
0 = \alpha_{b} \ \alpha_{c} \\
0 = \alpha_{b} \ \alpha_{ab} \\
\alpha_{b} = \alpha_{b} \ \alpha_{d} \\
\alpha_{b} = 0 \\
\alpha_{b} =  - q \ \alpha_{d} \ \alpha_{ab} \\
\alpha_{c} = \alpha_{c} \ \alpha_{d} \\
\alpha_{c} = \alpha_{c} \ \alpha_{d} + q^{-1} \ \alpha_{d} \ \alpha_{ab} \\
\alpha_{c} = - q \ \alpha_{d} \ \alpha_{ab} \\
\alpha_{d} = \alpha_{b} \ \alpha_{c} + (\alpha_{d})^2 \\
\alpha_{d} = - q^{-1} \ \alpha_{c} \ \alpha_{ab} \\
0 = \alpha_{b} \ \alpha_{ab} \\
0 = 0 \\
0 = 0 \\
\end{array} 
\right.
& \qquad \qquad 
\left\{
\begin{array}{lll}
0 = 0 \\
0 = \alpha_{d} \ \alpha_{ab} \\
0 = (\alpha_{b})^2\\
0 = (\alpha_{c})^2\\
0 = q^{-1} \ \alpha_{b} \ \alpha_{d}\\
(\alpha_{b})^2 = \alpha_{b} \ \alpha_{ab} \\
\alpha_{b} \ \alpha_{c} = \alpha_{b} \ \alpha_{ab} \\
\alpha_{b} \ \alpha_{c} = \alpha_{c} \ \alpha_{ab} \\
\alpha_{b} \ \alpha_{d} = \alpha_{c} \ \alpha_{d} \\
\alpha_{b} \ \alpha_{d} = \alpha_{d} \ \alpha_{ab} \\
\alpha_{b} \ \alpha_{ab} = \alpha_{c} \ \alpha_{ab} \\
(\alpha_{c})^2 = \alpha_{c} \ \alpha_{ab} \\
\alpha_{c} \ \alpha_{d} = \alpha_{d} \ \alpha_{ab} \\
(\alpha_{d})^2 = q \ \alpha_{c} \ \alpha_{ab}
\end{array} 
\right.
\end{align*}
which is obviously equivalent to
\begin{align*}
\left\{
\begin{array}{lll}
\alpha_{ac} = \alpha_{ab} \\
\alpha_{b} = 0 \\
\alpha_{c} = 0 \\
\alpha_{d} = 0.
\end{array} 
\right.
\end{align*}
Hence, in this case, for any $\alpha \in \Bbbk$, we have the partial action $\lambda_\alpha: \K \to \Bbbk$ as defined in \eqref{lda_alpha}.

\item If $\alpha_a \neq 0$, then \eqref{relacoes_K_red} is equivalent to
\begin{align*}
\left\{
\begin{array}{lll}
\alpha_{ac} = \alpha_{ab} \\
\alpha_{ab} = \alpha_{c} \\
\alpha_{c} = \alpha_{b} \\
\alpha_{a} = (\alpha_{a})^2 + (\alpha_{b})^2 \\
\alpha_{a} = \alpha_{a} \ \alpha_d - q \ (\alpha_{b})^2 \\
\alpha_{b} = \alpha_{a} \ \alpha_b + \alpha_b \ \alpha_{d} \\
\alpha_{b} = (1+q) \ \alpha_{a} \ \alpha_{b} \\
\alpha_{b} = q \ \alpha_{a} \ \alpha_{b} - q \ \alpha_{b} \ \alpha_{d} 
\end{array} 
\right.
& \qquad \qquad 
\left\{
\begin{array}{lll}
\alpha_{b} = (1-q) \ \alpha_{b} \ \alpha_{d} \\
\alpha_{d} = (\alpha_{b})^2 + (\alpha_{d})^2 \\
\alpha_d = \alpha_a \ \alpha_d + q \ (\alpha_b)^2 \\
(\alpha_{a})^2 = - q \ (\alpha_{b})^2 \\
\alpha_{a} \ \alpha_{b} = - q \ \alpha_{b} \ \alpha_{d} \\
\alpha_{a} \ \alpha_{d} = (\alpha_{b})^2\\
(\alpha_{d})^2 = q \ (\alpha_{b})^2
\end{array} 
\right.
\end{align*}
which in turn is equivalent to
\begin{align*}
\left\{
\begin{array}{lll}
\alpha_{ac} = \alpha_{b} \\
\alpha_{ab} = \alpha_{b} \\
\alpha_{c} = \alpha_{b} \\
\alpha_a = (1-q)/2 \\
\alpha_{d} \  = (1+q)/2 \\
(\alpha_{b})^2 = 1/2.
\end{array} 
\right.
\end{align*}

Hence, in this case, we obtain the two partial actions of $\K$ on $\k$ given by $\lambda_+$ and $\lambda_{-}$, for $\alpha_b \in \{\pm \sqrt{2}/{2}\}$, as defined in \eqref{lda_plus} and \eqref{lda_minus}. \qedhere
\end{itemize}
\end{proof}

\begin{prop}
    All partial actions of $\K$ on $\Bbbk$ determined in Theorem \ref{acoes_parciais_K} are symmetric.
\end{prop}

\begin{proof}
We need to check that such partial actions satisfy the symmetry condition \eqref{eqn_lda_sim}.
The global action (that is, the counit map) is always symmetric, and verifying that the partial action $\lambda_{\{1, a^2\}}^0$ satisfies \eqref{eqn_lda_sim} is an easy task, since $\lambda_{\{1, a^2\}}^0(a^2 v) = \lambda_{\{1, a^2\}}^0(v)$ and $\lambda_{\{1, a^2\}}^0(ab v) = \lambda_{\{1, a^2\}}^0(ac v) =0$, for all $v \in \B_\K$.

Thus, we will only focus on the partial actions $\lambda_+, \lambda_{-}$ and $\lambda_\alpha$, for any $\alpha \in \k$.
Let $\alpha\in \Bbbk$ and $\lambda \in \{\lambda_{+}, \lambda_{-}, \lambda_\alpha\}$.
We need to check if 
\begin{equation}\label{eqn_simetria_K_pma}
\lambda(u)\lambda(v) = \lambda(u_1 v) \lambda(u_2)
\end{equation}
holds for all $u, v \in \B_\K.$
Since $\lambda(1)=1$ and $\lambda(a^2)=0$, the above equalities clearly hold for all $u \in \{1, a^2, ac\}$ and $v \in \B_\K$.
Moreover, we also have $\lambda(b)=\lambda(c)$ and $\lambda(ab)=\lambda(ac)$.
So, if one denotes $\lambda(ac)=\alpha_{ab}$, $\lambda(ab)=\alpha_{ac}$, $\lambda(b)=\alpha_{c}$, $\lambda(c)=\alpha_{b}$, $\lambda(a)=\alpha_{a}$ and $\lambda(d)=\alpha_{d}$, when checking the equality \eqref{eqn_simetria_K_pma} for every $u \in \{a, b, c, d, ab \}$ and $v \in \B_\K$, we obtain exactly the same 40 equalities at \eqref{sist_K}. This shows that the partial action \(\lambda\) is indeed symmetric. 
\end{proof}

We now determine the partial coactions of \(\K\) on the base field \(\Bbbk\). Since \(\K\) is dual to the pointed Hopf algebra \(\A_4''\), any partial \(\K\)-coaction on \(\Bbbk\) can be obtained from a partial action of \(\A_4''\) on \(\Bbbk\), which are described in Table \ref{table}. An explicit Hopf algebra isomorphism \(\phi : \K \to (\A_4'')^*\) is given by
\begin{align*}
    1 &\mapsto 1^* + g^* + (g^2)^* + (g^3)^* & b &\mapsto \sqrt{2} \left(x^* + q(gx)^* - (g^2x)^* - q(g^3x)^*\right) \\
    a &\mapsto 1^* + qg^* - (g^2)^* - q(g^3)^* & ab &\mapsto q \sqrt{2} \left( x^* -  (gx)^* +  (g^2x)^* -  (g^3x)^*\right) \\
    a^2 & \mapsto 1^* -g^* + (g^2)^* - (g^3)^* & c & \mapsto -\sqrt{2} \left(x^* - q (gx)^* - (g^2 x)^* + q (g^3 x)^*\right) \\
    d & \mapsto 1^* - qg^* - (g^2)^* + q(g^3)^* & ac &\mapsto -q \sqrt{2} \left( x^* + (gx)^* + (g^2x)^* + (g^3x)^* \right)
\end{align*}
whose inverse is
\begin{align*}
    1^* &\mapsto \frac{1}{4} (1 + a + a^2 + d)& x^* &\mapsto \frac{\sqrt{2}}{8} (b - qab - c + q ac) \\
    g^* &\mapsto \frac{1}{4} (1 - qa - a^2 + qd) & (gx)^* &\mapsto -q \frac{\sqrt{2}}{8} ( b - ab + c - ac) \\
    (g^2)^* &\mapsto \frac{1}{4} (1 - a + a^2 - d) & (g^2 x)^* &\mapsto - \frac{\sqrt{2}}{8} (b + q ab - c -q ac) \\
    (g^3)^* &\mapsto \frac{1}{4} (1 + qa - a^2 -qd) & (g^3 x)^* &\mapsto q \frac{\sqrt{2}}{8} (b + ab + c + ac).    
\end{align*}

The partial coactions of \(\mathcal{K}\) on the base field are given by
\[\Bbbk \to \Bbbk \otimes \mathcal{K}, \quad 1 \mapsto 1 \otimes \phi^{-1}(\lambda),\]
where \(\lambda \in (\mathcal{A}_{4}'')^*\) is a partial action of \(\mathcal{A}_4''\) on \(\Bbbk\). Hence we find the partial coactions
\begin{align*}
    1 &\mapsto 1 \otimes \phi^{-1}(\varepsilon) = 1 \otimes 1, \\
    1 &\mapsto 1 \otimes \phi^{-1}\left(1^* + (g^2)^* + \alpha( x^* + (gx)^* + (g^2x)^* + (g^3x)^*)\right) \\
    &\qquad = 1 \otimes \left(\frac{1}{2}(1 + a^2) + \alpha q \frac{\sqrt{2}}{2} ac \right), \\
     1 &\mapsto 1 \otimes \phi^{-1}\left(1^* + \zeta ( x^* + (g^3x)^*)\right) \\
    &\qquad = 1 \otimes \left(\frac{1}{4}(1 + a + a^2 + d) +\zeta \frac{\sqrt{2}}{8} ((1 + q)b - (1 - q) c + 2q ac)\right),
\end{align*}
where \(\alpha, \zeta \in \Bbbk\) with \(\zeta = \pm q\).

\section{Partial smash products with the base field}

	Let $A$ be a partial $H$-module algebra via $\cdot : H \otimes A \longrightarrow A$.
    Then, the vector subspace $\underline{ A \# H}=(A\# H)(1_A\# 1_H)$ of \(A \# H\) is a unital algebra, called the \emph{partial smash product algebra of $A$ with $H$}. Adopting the notation $\underline{(x\#h)} = (x \# h)(1_A \# 1_H) = x(h_1\cdot 1_A)\# h_2$, the multiplication on \(\underline{A \# H}\) is induced by
	$$\underline{(x\#h)} \ \underline{(y\#g)}=\underline{x(h_1\cdot y)\# h_2g},$$
	for all $x,y\in A$, $h,g\in H$. For more details, see \cite{caenepeel2008partial}.

In particular, for a partial action of $H$ on $\k$ given by $\lambda: H \longrightarrow \k$, the partial smash product 
\begin{equation}
    \underline{\Bbbk \# H} = \{ \underline{1_\Bbbk \# h} \mid h \in H \} = \{ 1_\Bbbk \# \lambda(h_1) h_2 \mid h \in H \}
\end{equation}
can be identified with the subspace
\begin{equation} 
    H_\lambda = \{\lambda(h_1) h_2 \mid h \in H\} \subseteq H.
\end{equation}
Thus \(H_\lambda\) is a subalgebra of \(H\). It is even a right coideal subalgebra, because \(\Delta(\lambda(h_1)h_2) = \lambda(h_1) h_2 \otimes h_3 \in H_\lambda \otimes H\).

These observations immediately spark two questions:
\begin{enumerate}
    \item When is \(H_{\lambda}\) in fact a Hopf subalgebra of \(H\)? In that case, \(H_\lambda\) is called a \textit{\(\lambda\)-Hopf algebra}.
    \item Which right coideal subalgebras of \(H\) are of the form \(H_\lambda\) for some partial action \(\lambda : H \to \Bbbk\)?
\end{enumerate}

\subsection{\texorpdfstring{$\lambda$-Hopf}{Lambda-Hopf} algebras}\label{lambda_Hopf}

Question (1) above was studied in \cite{corresponding}*{Sect. 4}. It turns out that in general, the subalgebra $H_\lambda$ is not a Hopf subalgebra of $H$. 
On the other hand, since $\varepsilon_H$ is a partial (in fact global) action of $H$ and $H_\varepsilon = H$, every Hopf algebra $H$ is an $\varepsilon$-Hopf algebra.
We recall the following characterizations of \(\lambda\)-Hopf algebras from \cite{corresponding}.

\begin{prop}[\cite{corresponding}*{Thm. 4.6}]\label{carac}
	Let $H$ be a finite-dimensional Hopf algebra and $\lambda: H \longrightarrow \k$ a partial action.
	Then, $H_\lambda $ is a Hopf subalgebra of $H$ if and only if
	\begin{align}\label{eq_carac}
		\lambda(h_1)h_2 = \lambda(h_1)h_2\lambda(h_3)
	\end{align}
	for all $ h \in H$.
\end{prop}

 \begin{prop}[\cite{corresponding}*{Cor. 4.9}]\label{cor_carac_lambdahopf}  Let $H$ be a finite-dimensional Hopf algebra and   $\lambda: H \longrightarrow \k$ a partial action. If $h\in P_{g,h}(H)\backslash\k\{g-h\}$ such that $\lambda(g)\neq\lambda(h)$, then $H_\lambda$ is not a  Hopf subalgebra of $H$.     
 \end{prop}

 \begin{prop}[\cite{corresponding}*{Prop. 4.11}]\label{cond_mais_forte}
	Let $\lambda: H \longrightarrow \k$ be a partial action.
	If $\lambda(h_1)h_2=h_1 \lambda(h_2)$ for all $h \in H$, then $\lambda(h_1)h_2=\lambda(h_1) h_2 \lambda(h_3)$ for all $h \in H$.
	In this case, if $H$ is finite-dimensional, then
	$H_\lambda$ is a Hopf subalgebra of $H$.
 \end{prop}

The $\lambda$-Hopf algebras of group algebras are well-understood: let $N$ be any subgroup of a group $G$, and consider the one-dimensional partial action $\lambda_N$ of $\k G$ (see Example \ref{exemplo_parcialgrupo}).
Since $(\Bbbk G)_{\lambda_N}\cong \Bbbk N$ as Hopf algebras, every partial action $\lambda_N$ of $\k G$ determines a $\lambda_N$-Hopf algebra. Conversely, every Hopf subalgebra of \(\k G\) is of the form \(\k N\) for some subgroup \(N\) of \(G\), so it is a \(\lambda\)-Hopf algebra. 

\medbreak

Now, let $G$ be a finite group. 
Recall that there is a correspondence between the subgroups of $G$ and one-dimensional partial actions of $\left(\k G\right)^{\ast}$ (see Example \ref{exemplo_parcialdualgrupo}). It turns out that a partial action \(\lambda : (\Bbbk G)^* \to \Bbbk\) induces a \(\lambda\)-Hopf algebra exactly when it is associated with a normal subgroup of \(G\).
 
\begin{prop}
   Let $N$ be a subgroup of a finite group $G$, and consider the partial action $\lambda^N:(\Bbbk G)^{\ast} \to \Bbbk$. Then ${\left(\left(\Bbbk G\right)^{\ast}\right) }_{\lambda^N}$ is a Hopf algebra if and only if $N$ is a normal subgroup of $G$.
   In this case, ${\left(\left(\Bbbk G\right)^{\ast}\right) }_{\lambda^N} \cong \left(\Bbbk \left[\frac{G}{N}\right]\right)^*$ as Hopf algebras.
\end{prop}
\begin{proof}
Let $H:=\left(\Bbbk G\right)^{\ast}=\spn_{\k}\{g^* \mid g \in G\}$, where $\{g \mid g \in G\}$ is the canonical basis of $\k G$ and $\{g^* \mid g \in G\}$ its dual basis, that is, $g^\ast: \k G \to \k$ is given by $g^\ast (h)= \delta_{g,h}$, for all $h \in G$.
Recall that $\Delta(g^*)=\sum_{h\in G}h^*\otimes (h^{-1}g)^*$. 

Let $N$ be a subgroup of $G$ and consider the partial action $\lambda^N: H \to \k$. 
Then, on the one hand,
\begin{equation} 
\label{eq:lambda^N}
\lambda^N\left({g^\ast}_1\right) {g^\ast}_2 = \sum_{h \in G} \lambda^N\left(h^\ast\right)\left({h^{-1}g}\right)^\ast = \frac{1}{|N|} \sum_{h \in N} \left({h^{-1}g}\right)^\ast = \frac{1}{|N|} \sum_{t \in Ng} t^\ast, \end{equation}
and, on the other hand,
$${g^\ast}_1 \lambda^N\left({g^\ast }_2\right) = \sum_{h \in G} \lambda^N\left(\left({h^{-1}g}\right)^\ast\right) h^{\ast}  = \frac{1}{|N|} \sum_{h \in gN} h^\ast, $$
where the last equality holds because \(\lambda^N(({h^{-1}g})^\ast) = 1\) if and only if \(h^{-1}g \in N\) (i.\,e.\ \(h \in gN\)), and zero else. 
Thus, if $N$ is a normal subgroup of $G$, it follows from Proposition \ref{cond_mais_forte} that ${\left(\left(\Bbbk G\right)^{\ast}\right) }_{\lambda^N}$ is a Hopf algebra.

Reciprocally, we assume that $\lambda({g^\ast}_1) {g^\ast}_2 =\lambda({g^\ast}_1 ){g^\ast}_2 \lambda({g^\ast}_3)$ for all $g\in G$ in light of Proposition \ref{carac}. On the one hand, \eqref{eq:lambda^N} holds. 
On the other hand, 
\begin{align*}
 \lambda^N({g^\ast}_1) {g^\ast}_2 \lambda^N({g^\ast}_3) & = \sum_{l,h\in G}\lambda^N(l^*)\lambda^N((h^{-1}g)^*)(l^{-1}h)^*\\
 & =  \frac{1}{|N|^2} \sum_{l\in N,h\in gN} (l^{-1}h)^*.
\end{align*}
Comparing with \eqref{eq:lambda^N}, we conclude that for all $l\in N$ and $h\in gN$, we have that $l^{-1}h\in Ng$, i.\,e.\ \(h \in l N g = Ng\). Therefore, $N$ is a normal subgroup of $G$. 

To prove the last claim, write \(G = \bigsqcup_{i = 1}^n N g_i\) for some representatives \(g_i\) of the cosets in \(G/N\).
Note that
\begin{align*}
    H_\lambda &=\spn_\k\left\{\lambda({g^*}_1){g^*}_2  \mid  g\in G\right\}=\spn_\k\left\{\frac{1}{N}\sum_{h\in N}(h^{-1}g)^* \mid g\in G\right\} \\
    &= \spn_\k\left\{\frac{1}{N}\sum_{h\in N}(h^{-1}g_i)^* \mid i = 1, \dots n \right\},
\end{align*}
so that the map 
\[\varphi : \left(\Bbbk \left[\frac{G}{N}\right]\right)^*\to H_\lambda, \quad \left(\overline{g_i}  \right)^{\ast} \mapsto \sum_{h\in N}\frac{1}{N}\left(h^{-1}g_i\right)^*\]
is a bijection. A direct check shows that it is in fact an isomorphism of Hopf algebras. 
\end{proof}

\medbreak

In \cite{corresponding}, it is shown that for the pointed Hopf algebras $\mathcal{A}_2, \mathcal{A}_4^\prime,$ and $ \mathcal{A}_4^{\prime \prime}$, only the global action \(\lambda = \varepsilon\) produces a $\lambda$-Hopf algebra;
for $\mathcal{A}_{4,q}^{\prime \prime \prime}$ and $\mathcal{A}_{2,2}$, each of them has a unique one-dimensional genuine partial action $\lambda$ that determines a $\lambda$-Hopf algebra: $(\mathcal{A}_{4,q}^{\prime \prime \prime})_{\lambda_{\{1, g^2\}}}\cong \mathbb{H}_4 \cong (\mathcal{A}_{2,2})_{\lambda_{\{1, g\}}}$.
For details, refer to \cite{corresponding}*{\S 4.3}. 

\medbreak

Thus, for 8-dimensional Hopf algebras, it only remains to analyze the situation for $\K$ and $\A$.

\begin{prop}
Consider the Hopf algebra $\K$.
Then, there exists a unique partial action $\lambda: \K \to \k$, $\lambda \neq \varepsilon$, such that $\K_\lambda$ is a Hopf algebra.
In this case, $\K_\lambda \cong \mathbb{H}_4$.
\end{prop}

\begin{proof} By Theoreom \ref{acoes_parciais_K},  all partial actions of $\K$ on $\Bbbk$ are $\varepsilon$, $\lambda_{\{1, a^2\}}^0$, $\lambda_{+}$, $\lambda_{-}$ and $\lambda_\alpha$, for all $\alpha \in \k$.
Since for $\lambda\in \{\lambda_{+}, \lambda_{-},\lambda_\alpha\} $, we have $ab \in P_{1, a^2}(\K)$ and $\lambda(1)\neq \lambda(a^2)$, we conclude that $\K_\lambda$ is not a Hopf algebra in these cases by Proposition \ref{cor_carac_lambdahopf}.

However, for $\lambda=\lambda_{\{1, a^2\}}^0$, it is clear that $\lambda(h_1)h_2=h_1 \lambda(h_2)$ holds for all $h\in\B_\K$.
Therefore, by Proposition \ref{cond_mais_forte}, $\K_\lambda$ is a Hopf subalgebra of $\K$.
Moreover, we get
        $\K_\lambda = \spn_\k\left\{ a^2, ab \right\},$
    and since $(a^2)^2 = a^4 = 1,\ (ab)^2 = 0,$ and $(ab)a^2 = - ac = - a^2 (ab)$, it follows that $\K_\lambda \cong \mathbb{H}_4$.
\end{proof}

Finally, for the Kac-Paljutkin algebra $\A$, recall that all one-dimensional partial actions are $\varepsilon$, $\lambda_N^0$ and $\lambda_\zeta$, for each subgroup $N$ of $G(\A)=C_2\times C_2$ and $ \zeta \in \{q, -q \}$ (see Theorem \ref{acoes_parciais_A}). Then,

\begin{prop}
Consider the Hopf algebra $\A$ and $\lambda: \A \to \k$ a partial action.
Then,
\begin{itemize}
\item if $\lambda = \lambda_\zeta$, then $\A_{\lambda_{\zeta}}$ is not a Hopf subalgebra of $\A$;
\item if $\lambda=\lambda_N^0$, for a subgroup $N$ of $G(\A)$, then the algebra $\A_{\lambda_N^0}$ is a Hopf subalgebra of $\A$. In this case, $\A_{\lambda_N^0}\cong \k N$.
\end{itemize}
\end{prop}

\begin{proof} 
First, for $\lambda_\zeta$, condition \eqref{eq_carac} is not satisfied for the element $x\in \A$.
Indeed, recall that $\Delta(x) = 
\tfrac{1}{2}(x\otimes x+x\otimes gx+hx\otimes x-hx\otimes gx)$; then, on the one hand,
\begin{equation*}
 \lambda_\zeta(x_1)x_2 = \tfrac{1}{2}(x+\zeta gx),   
\end{equation*}
and, on the other hand,
\begin{equation*}
 \lambda_\zeta(x_1)x_2\lambda_\zeta(x_3)= \tfrac{1}{2}(x+\zeta gx+\zeta hx+ghx).   
\end{equation*}
Thus, by Proposition \ref{carac}, the algebra $\A_{\lambda_\zeta}$ is not a Hopf subalgebra of $\A$.

Now, for each subgroup $N$ of $G(\A)\cong C_2 \times C_2$, it is easy to verify that $\lambda_N^0(h_1)h_2=h_1 \lambda_N^0(h_2)$ holds for all $h\in\B_\A$.
Therefore, $\A_{\lambda_N^0}$ is a Hopf subalgebra of $\A$ by Proposition \ref{cond_mais_forte}.
Moreover, $$\A_{\lambda_N^0} = \spn_\k\{\lambda_N^0(h_1)h_2 \mid h \in \B_\A\} = \spn_\k\{\lambda_N^0(g)g \mid g \in N\} \cong \k N.$$
In particular, 
$\A_{\lambda_{\{1\}}^0}\cong \Bbbk$, $\A_{\lambda_{G(\A)}^0}\cong \Bbbk(C_2\times C_2)$ and $\A_{\lambda_{\{1,g\}}^0}\cong \A_{\lambda_{\{1,h\}}^0}\cong \A_{\lambda_{\{1,gh\}}^0}\cong\Bbbk C_2$.
\end{proof}

    \subsection{Which right coideal subalgebras are partial smash products?}\label{right coideal subalgebras}

    The goal of this section is to determine which right coideal subalgebras of \(H\) are of the form \(\underline{\k \# H}\) for some partial action \(\lambda : H \to \k\). 

 In the next theorem, a necessary condition for a right coideal subalgebra of a pointed Hopf algebra to be a partial smash product is obtained. Its proof makes use of the coradical filtration \((H_n)_{n \geq 0}\) of \(H\). The coradical \(H_0\) is defined as the sum of all simple subcoalgebras of \(H\).  The higher terms of the coradical filtration are the coalgebras recursively defined by \[H_n = \Delta^{-1}(H \otimes H_0 + H_{n - 1} \otimes H).\] They satisfy 
\(\Delta(H_n) \subseteq \sum_{i = 0}^n H_i \otimes H_{n - i}\)
and 
\(H = \bigcup_{n \geq 0} H_n.\)

A Hopf algebra is pointed exactly when the coradical \(H_0\) is a group algebra. In that case, we also have that \(H_i H_j \subseteq H_{i + j}\) for all \(i, j \in \mathbb{N}\).

\begin{theorem}
    Let \(H\) be a pointed Hopf algebra and \(B = \underline{\Bbbk \# H} = \{\lambda(h_1) h_2 \mid h \in H\} \subseteq H\). Then \(B\) generates \(H\) as an \(H_0\)-module. 
\end{theorem}
\begin{proof}
    We denote the \(H_0\)-submodule of \(H\) generated by \(B\) as \(H_0B\). We show by induction that for each \(n \in \mathbb{N},\) \(H_n \subseteq H_0 B\). It will follow that \(H = \bigcup_{n \in \mathbb{N}} H_n \subseteq H_0 B\).

    Since \(1_H \in B,\) certainly \(H_0 \subseteq H_0 B\). Suppose now that \(H_n \subseteq H_0 B\). By \cite{radford}*{Eq. 4.5 and Prop. 4.3.1}, \(H_{n + 1}\) is generated as a vector space by elements \(b\) for which 
    \[\Delta(b) = g \otimes b + b \otimes h + \sum_i y_i \otimes z_i\]
    where \(g\) and \(h\) are grouplike, and \(y_i\) and \(z_i\) are in \(H_n\). Then \(c = g^{-1} b\) satisfies
    \[
    \Delta(c) = 1_H \otimes c + c \otimes g^{-1} h + \sum_i y'_i \otimes z'_i
    \]
    with \(y'_i, z'_i \in H_n\). By the definition of \(B,\)
    \[
    \lambda(c_1) c_2 = c + \lambda(c) g^{-1} h + \sum_i \lambda(y'_i) z'_i \in B.
    \]
    Since \(g^{-1} h \in H_0 \subseteq H_0 B\) and \(z'_i \in H_n \subseteq H_0 B,\) it follows that \(c \in H_0 B\) too. Then \(b = gc \in H_0 B\) as well, so \(H_{n + 1} \subseteq H_0 B\). This completes the proof by induction.  
\end{proof}

\begin{cor}
Let $H$ be a pointed Hopf algebra and $B \subseteq H_0$ a subalgebra.
If \(H_0 \neq H\), then $B$ cannot be written as a partial smash product \(\underline{\k \# H}\).
\end{cor}

\medbreak
    
Let now \(H\) be any finite-dimensional Hopf algebra. Our next goal is to determine a necessary and sufficient condition for right coideal subalgebras to be of the form \(\underline{\k \# H}\).
    Let us recall first some facts about right coideal subalgebras in general.

    Let \(B\) be a right coideal subalgebra of \(H\) and denote the inclusion by \(\iota : B \to H\). Then the restriction map \(\iota^* : H^* \to B^*\) is a right \(H^*\)-linear coalgebra projection. The space
    \begin{align}
        r_H(B) = {^{\mathrm{co} B^*}} H^* &= \{\varphi \in H^* \mid \iota^*(\varphi_1) \otimes \varphi_2 = \iota^*(1_{H^*}) \otimes \varphi\} \label{def:rH1} \\
        &= \{\varphi \in H^* \mid \varphi(bh) = \varepsilon(b) \varphi(h) \text{ for all } b \in B, h \in H\} \label{def:rH2}
    \end{align}
    is a right coideal subalgebra of \(H^*\): it is a right coideal because for any \(\varphi \in r_H(B)\) and \(h \in H\), \(\varphi_1 \varphi_2(h) \in r_H(B)\) thanks to \eqref{def:rH1}. Indeed, \(\iota^*((\varphi_1)_1) \otimes (\varphi_1)_2 \varphi_2(h) = \iota^*(\varphi_1) \otimes (\varphi_2)_1 (\varphi_2)_2(h) = \iota^*(1_{H^*}) \otimes \varphi_1 \varphi_2(h)\). Since for \(\varphi, \psi \in r_H(B)\), 
    \[(\varphi \psi)(bh) = \varphi(b_1 h_1) \psi(b_2 h_2) = \varepsilon(b_1) \varphi(h_1) \psi(b_2 h_2) = \varphi(h_1) \psi(bh_2) = \varepsilon(b) (\varphi\psi)(h)\]
    for all \(b \in B\) and \(h \in H\) by \eqref{def:rH2}, \(r_H(B)\) is also a subalgebra of \(H^*\).
    
    If we write \(B^+ = B \cap \ker \varepsilon\), then \({^{\mathrm{co} B^*}} H^* \cong (H / B^+H)^*\) (cf.\ \cite{shimizu}*{Lem. 2.2}). In fact, the assignment \(r_H : B \mapsto {^{\mathrm{co} B^*}} H^*\) provides a bijection between the right coideal subalgebras of \(H\) and those of \(H^*\), as is proved in the following proposition (see also \cite{shimizu}*{Sect. 2.3}).

    \begin{prop}
    \label{prop:rHbij}
        Denote by \(\phi : H \to H^{**}\) the canonical isomorphism. Then \(r_{H^*} \circ r_H(B) = \phi(B)\) for every right coideal subalgebra \(B \subseteq H\).
    \end{prop}
    \begin{proof}
        We write \(\overline{B} = r_H(B) = {^{\mathrm{co} B^*}} H^*\). 
        For \(h \in H\), the element \(\phi(h) \in H^{**}\) is in \({^{\mathrm{co} \overline{B}^*}} H^{**}\) if and only if 
        \(\pi(\mathrm{ev}_{h_{1}}) \otimes \mathrm{ev}_{h_2} = \pi(1_{H^{**}}) \otimes \mathrm{ev}_h,\)
        where \(\pi : H^{**} \to \overline{B}^*\) is the restriction. This is equivalent to 
        \begin{equation}
            \label{eq:rH*}
            \psi(h_1) h_2 = \psi(1_H) h
        \end{equation}
        for all \(\psi \in \overline{B}\).

        Take now \(b \in B\). Then for any \(\psi \in \overline{B}\),
        \(\psi(b_1) b_2 \overset{\eqref{def:rH2}}{=} \varepsilon(b_1) \psi(1_H) b_2 = \psi(1_H) b\)
        by definition of \(\overline{B}\). This shows that \(\phi(B) \subseteq {^{\mathrm{co} \overline{B}^*}} H^{**} = r_{H^*} \circ r_H(B)\). 

        As a consequence of \cite{skryabin}*{Thm. 6.1} (see also \cite{shimizu}*{Lem. 2.1}), 
        \[\dim(B) \dim(r_H(B)) = \dim(B) \dim(H/B^+H) = \dim (H),\]
        so \(\dim(B) = \dim(r_{H^*} \circ r_H(B))\) and we conclude that \(\phi(B) = r_{H^*} \circ r_H(B)\).
    \end{proof}

    \begin{prop}
    \label{prop:smallest_rcs}
        Let \(\lambda : H \to \k\) be a partial action. Then the smallest right coideal subalgebra of \(H^*\) containing \(\lambda\) is \(r_H(\underline{\k \# H})\).
    \end{prop}
    \begin{proof}
        Write \(B = \underline{\k \# H} = \{\lambda(h_1) h_2 \mid h \in H\} \subseteq H\). Let \(C\) be the smallest right coideal subalgebra of \(H^*\) containing \(\lambda\). Since 
        \[\lambda(\lambda(h_1)h_2 k) = \lambda(h_1) \lambda(h_2k) = \lambda(h) \lambda(k) = \varepsilon(\lambda(h_1) h_2) \lambda(k),\]
        \(\lambda \in r_H(B)\) by \eqref{def:rH2}. Hence \(C \subseteq r_H(B)\).
        Let \(\phi : H \to H^{**}\) be the canonical isomorphism and take \(\phi(x) \in r_{H^*}(C)\), where \(x \in H\). Then \(\psi(x_1) x_2 = \psi(1_H) x\) for all \(\psi \in C\) by \eqref{eq:rH*}. In particular
        \[\lambda(x_1) x_2 = \lambda(1_H) x = x,\]
        which shows that \(x \in B\). Hence \(r_{H^*}(C) \subseteq \phi(B)\). Since \(r_{H}\) is inclusion-reversing, we obtain that \(r_H(B) \subseteq C\) by Proposition \ref{prop:rHbij}.
    \end{proof}

    \begin{theorem}
        \label{semisimple}
        Let \(B \subseteq H\) be a right coideal subalgebra. Then \(B = H_\lambda\) for some partial action \(\lambda : H \to \k\) if and only if \(r_H(B) = {^{\mathrm{co} B^*}} H^*\) contains a normalized right integral. This is in particular the case if \(r_H(B)\) is semisimple. 
    \end{theorem}
    \begin{proof}
        Suppose that there exists a partial action \(\lambda : H \to \k\) such that
        \(B = H_\lambda = \underline{\k \# H} = \{\lambda(h_1) h_2 \mid h \in H\}\). 
        By Proposition \ref{prop:smallest_rcs}, \(\lambda \in r_H(B)\).
        Moreover, for any \(\psi \in r_H(B)\) we have
        \begin{align*}
            (\lambda \psi)(h) = \lambda(h_1) \psi(h_2) = \psi(b) = \varepsilon(b) \psi(1_H) = \varepsilon_{H^*}(\psi) \lambda(h),
        \end{align*}
        so \(\lambda\) is a right integral in \(r_H(B)\). It is normalized since \(\varepsilon_{H^*}(\lambda) = \lambda(1_H) = 1\).

        Conversely, suppose that \(r_H(B)\) contains a normalized right integral \(\lambda\). Then \(\lambda(1_H) = 1\) and \(\lambda\) is idempotent, so \(\lambda(h_1) \lambda(h_2) = \lambda(h)\) for all \(h \in H\). For any \(k \in H,\) the functional \(\psi_k(h) = \lambda(hk)\) is in \(r_H(B)\) because
        \[\psi_k(bh) = \lambda(bhk) = \varepsilon(b) \lambda(hk) = \varepsilon(b) \psi_k(h).\]
        Since \(\lambda\) is a right integral in \(r_H(B)\), we have
        \[\lambda(h_1) \lambda(h_2k) = \lambda(h_1) \psi_k(h_2) = (\lambda \psi_k)(h) = \lambda(h) \psi_k(1_H) = \lambda(h) \lambda(k).\]
        We can conclude that \(\lambda\) is a partial \(H\)-action on \(\k\). 

        Since \(r_H(\underline{\k \# H})\) is the smallest right coideal subalgebra of \(H^*\) containing \(\lambda\) by Proposition \ref{prop:smallest_rcs}, we conclude that \(r_H(\underline{\k \# H}) \subseteq r_H(B)\) and hence \(B \subseteq \underline{\k \# H}\).
        On the other hand, for any \(h \in H,\) the element \(c= \lambda(h_1) h_2\) satisfies
        \[\psi(c_1) c_2 = \lambda(h_1) \psi(h_2) h_3 = \psi(1_H) \lambda(h_1) h_2 = \psi(1_H) c,\]
        for all \(\psi \in r_H(B)\). Hence \(c \in B\) by \eqref{eq:rH*}. This shows that \(B = H_\lambda\).

        Finally, if \(r_H(B)\) is semisimple, then \(r_H(B) \cap \ker \varepsilon_{H^*} \subseteq r_H(B)\) has a one-dimensional complement as right \(r_H(B)\)-module. This complement contains necessarily a normalized right integral. 
    \end{proof}

    \begin{cor}
    \label{cor:cosemi}
        Let \(L \subseteq H\) be a Hopf subalgebra. Then \(L = H_\lambda\) for some partial action \(\lambda : H \to \k\) if and only if the coalgebra \(H/L^+ H\) is cosemisimple, where \(L^+ = L \cap \ker \varepsilon = \ker \varepsilon_{|L}\).
    \end{cor}
    \begin{proof}
        If \(H/L^+H\) is cosemisimple, then \((H/L^+H)^* \cong r_H(L)\) is semisimple and we are in the situation of Proposition \ref{semisimple}.
        Conversely, if \(L = H_\lambda,\) then \(L\) is in particular a right coideal subalgebra such that \(r_H(L)\) contains a normalized right integral. Since \(L\) is stable under the square of the antipode, \(r_H(L)\) is too: indeed, for any \(\psi \in r_H(L),\)
        we have
        \[S^2(\psi)(lh) = \psi(S^2(l)S^2(h)) = \varepsilon(S^2(l)) \psi(S^2(h)) = \varepsilon(l) S^2(\psi)(h), \]
        so \(S^2(\psi) \in r_H(L)\). By \cite{Frobenius}*{Sect. 4}, \(r_H(L)\) is semisimple, hence \(H/L^+ H\) is cosemisimple. 
    \end{proof}

    If \(H\) is a finite-dimensional cosemisimple Hopf algebra, then \(H^*\) is semisimple, and so are all right coideal subalgebras of \(H^*\) by \cite{Frobenius}*{Sect. 4}. Hence the hypothesis of Theorem \ref{semisimple} is always satisfied, and every right coideal subalgebra of \(H\) is a partial smash product \(\underline{\k \# H}\). This shows that there is a one-to-one correspondence between right coideal subalgebras of \(H\) and partial \(H\)-actions on \(\k\). This applies in particular to finite-dimensional group algebras and their duals, and to the Kac-Paljutkin algebra \(\A\).

    \begin{exa}
        We apply Corollary \ref{cor:cosemi} to the case where \(H = H_1 \otimes H_2\) is the tensor product of two Hopf subalgebras. For \(L = H_1 \otimes 1_{H_2} \subseteq H\), we obtain
        \[L^+ H = \ker \varepsilon_{H_1} \otimes H_2,\]
        so \(H/L^+ H \cong H_2\) as coalgebras. We conclude that \(L\) is realized as a partial smash product \(\underline{\Bbbk \# H}\) if and only if \(H_2\) is cosemisimple. In this case the partial action is given by \(\lambda = \varepsilon_{H_1} \otimes \varphi\), where \(\varphi\) is the normalized integral of the semisimple Hopf algebra \(H_2^*\).
    \end{exa}

\subsection*{Author Contributions and AI disclaimer} All authors conceived, drafted, and revised the manuscript, and none of them used AI or LLM tools in any aspect of this research or in the writing of this manuscript.
The authors are Matheus Castelo, Leonardo Duarte Silva, William Hautekiet, and Grasiela Martini.
The work was carried out through joint discussions involving all or some of the authors.
It is impossible to determine exactly who performed which specific task.

\bibliographystyle{abbrv}

\begin{thebibliography}{99}	
	\bibitem{Alvares_Alves_Batista} E. Alvares, M. M. Alves and E. Batista, \emph{Partial Hopf module categories}, Journal of Pure and Applied Algebra 217 (8) (2013), 1517-1534.

	\bibitem{enveloping} M. M. Alves and E. Batista, \emph{Enveloping Actions for Partial Hopf Actions}, Communications in Algebra 38 (8) (2010), 2872-2902.

    \bibitem{parcorep} M.~M.~S. Alves, E. Batista and J. Vercruysse, \emph{Partial corepresentations of Hopf algebras,} J. Algebra  577 (2021), 74--135.

    \bibitem{Arthur} M. M. Alves and A. R. A. Neto, \emph{On partial representations of pointed Hopf algebras}, arXiv e-prints arxiv:2502.03642, (2025), \url{https://arxiv.org/abs/2502.03642}.    

    \bibitem{BHSV} E. Batista, W. Hautekiet, P. Saracco and J. Vercruysse, \emph{Towards a classification
    of simple partial comodules of Hopf algebras}, J. Algebra 664 (2025), 312-347

	\bibitem{classifying} M. Beattie and G. Garcia, \emph{Classifying Hopf algebras of a given dimension}, American Mathematical Society 585 \textbf{In} Hopf Algebras and Tensor Categories - Contemporary Mathematics  (2013), 125-152.
	
	\bibitem{caenepeel2008partial} S. Caenepeel and K. Janssen, \emph{Partial (co)actions of Hopf algebras and partial Hopf-Galois theory}, Communications in Algebra 36 (2008), 2923-2946.

	\bibitem{guris} F. Castro, A. Paques, G. Quadros and A. Sant'Ana, \emph{Partial actions of weak Hopf algebras: smash products, globalization and Morita theory}, Journal of Pure and Applied Algebra 29 (2015), 5511-5538.

    \bibitem{DT} M.-C. David and N.~M. Thi\'ery, \emph{Exploration of finite-dimensional Kac algebras and lattices of intermediate subfactors of irreducible inclusions}, J. Algebra Appl.  10 (2011), no.~5, 995--1106.

	\bibitem{Dokuchaev_survey} M. Dokuchaev, \emph{Recent developments around partial actions}, São Paulo Journal of Mathematical Sciences 13 (2019), 195-247.

	\bibitem{Dokuchaev_exel} M. Dokuchaev and R. Exel, \emph{Associativity of crossed products by partial actions, enveloping actions and partial representations}, Transactions of the American Mathematical Society 357 (2005), 1931-1952.

	\bibitem{paques_ferrero_dokuchaev} M. Dokuchaev, M. Ferrero and A. Paques, \emph{Partial actions and Galois theory}, Journal of Pure and Applied Algebra 208 (2007), 77-87.

	\bibitem{exel1994circle} R. Exel, \emph{Circle actions on $C^\ast$-algebras, partial automorphisms and generalized Pimsner-Voiculescu exact sequences}, Journal of Functional Analysis 122 (02) (1994), 361-401.

 \bibitem{Frobenius} D. Fischman, S. Montgomery and H.-J. Schneider, \emph{Frobenius extensions of subalgebras of Hopf algebras,} Trans. Amer. Math. Soc. {349} (1997), no.~12, 4857--4895.

\bibitem{FMS} G. Fonseca, G. Martini and L. D. Silva, \emph{Partial (co) actions of {T}aft and {N}ichols {H}opf algebras on their base fields}, International Journal of Algebra and Computation 31 (07) (2021), 1471-1496.

\bibitem{FMS2} G. Fonseca, G. Martini and L. D. Silva, \emph{Partial (co)actions of {T}aft and {N}ichols {H}opf algebras on algebras}, Journal of Pure and Applied Algebra 228 (01) (2024), 107455.

\bibitem{Ore} J. M. J. Giraldi, G. Martini and L. D. Silva, \emph{On partial actions of Hopf-Ore extensions}, Bull. Braz. Math. Soc. 57 (03) (2026), No. 30, 1-33.

\bibitem{corresponding} G. Martini, A. Paques and L. D. Silva. \emph{Partial actions of a Hopf algebra on its base field and the corresponding partial smash product algebra}, Journal of Algebra and Its Applications 22 (06) (2023), 2350140.

\bibitem{masuoka} A. Masuoka, \emph{Semisimple Hopf algebras of dimension $6,8$}, Israel J. Math. 92 (1995), no.~1-3, 361--373.

\bibitem{shimizu} K. Shimizu and R. Sugitani, \emph{Coideal subalgebras of quantum $SL_{2}$ at roots of unity}, J. Pure Appl. Algebra 229 (2025), no.~5, Paper No. 107923, 34 pp.

\bibitem{skryabin} S. Skryabin, \emph{Projectivity and freeness over comodule algebras,} Trans. Amer. Math. Soc. 359 (2007), no.~6, 2597--2623. 

\bibitem{stefan} D. \c Stefan, \emph{Hopf algebras of low dimension}, J. Algebra  211 (1) (1999), 343--361.

\bibitem{radford} D. Radford, \emph{Hopf Algebras}, \textbf{In} K \& E series on knots and everything (2011).
    
	\end{thebibliography}

\end{document}